\documentclass{article}

\usepackage{geometry}
\usepackage[normalem]{ulem}
\usepackage{amsfonts}
\usepackage{graphicx}
\usepackage{subfig}
\usepackage{epstopdf}
\usepackage[ruled,algosection]{algorithm2e}
\ifpdf
  \DeclareGraphicsExtensions{.eps,.pdf,.png,.jpg}
\else
  \DeclareGraphicsExtensions{.eps}
\fi

\newcommand{\primal}{x}
\newcommand{\dualit}{y}
\newcommand{\momentumone}{u}
\newcommand{\momentumtwo}{v}
\newcommand{\taustep}{\tau}

\usepackage{lipsum}
\usepackage{amsfonts}
\usepackage{graphicx}
\usepackage{epstopdf}

\usepackage{amsthm}
\newtheorem{theorem}{Theorem}[section]
\newtheorem{lemma}[theorem]{Lemma}
\newtheorem{corollary}[theorem]{Corollary}
\newtheorem{remark}[theorem]{Remark}

\usepackage{amsmath}
\numberwithin{equation}{section}

\usepackage{amsopn}

\usepackage{url}            
\usepackage{booktabs}       
\usepackage{nicefrac}       
\usepackage{microtype}      

\usepackage{ucs} 

\usepackage{bm}

\usepackage{amsmath}
\newcommand{\defeq}{\overset{\text{\tiny def}}{=}}

\usepackage{algpseudocode}

\usepackage{tikz}

\tikzset{
  treenode/.style = {shape=rectangle, rounded corners,
                     draw, align=center,
                     top color=white,
                     bottom color=blue!20},
  root/.style     = {treenode, font=\Large,
                     bottom color=red!30},
  env/.style      = {treenode, font=\ttfamily\normalsize},
  dummy/.style    = {circle,draw}
}

\usepackage{color}

\newcommand{\prox}{\textnormal{prox}}
\DeclareMathOperator*{\argmin}{arg\,min}

\usepackage{subfig}

\usepackage{array}
\newcolumntype{C}[1]{>{\centering\arraybackslash}p{#1}}

\algrenewcommand\algorithmicrequire{\textbf{Input:}}
\algrenewcommand\algorithmicensure{\textbf{Output:}}

\title{Practical Acceleration of the Condat--V\~u Algorithm}
\date{} 

\author{Derek Driggs\thanks{Department of Applied Mathematics and Theoretical Physics, University of Cambridge, Cambridge, UK.}
\and Matthias J. Ehrhardt\thanks{Department of Mathematical Sciences, University of Bath, Bath, UK}
\and Carola-Bibiane Sch\"onlieb\thanks{Department of Applied Mathematics and Theoretical Physics, University of Cambridge, Cambridge, UK.} \and Junqi Tang\thanks{School of Mathematics, University of Birmingham, Birmingham, UK 
  (j.tang.2@bham.ac.uk).}}

\usepackage[english]{babel}
\usepackage{enumerate}
\usepackage{amsmath,amsfonts,amssymb}
\usepackage{color}
\usepackage{graphicx}
\usepackage{tabularx}
\usepackage{stmaryrd}                         
\usepackage{multicol}
\usepackage{hyperref}                           

\usepackage{tikz}
\usetikzlibrary{arrows,shapes}

\tikzstyle{every picture}+=[remember picture]
\tikzstyle{na} = [baseline=-.5ex]

  \usetikzlibrary{arrows}
  \usetikzlibrary{automata}
  \usetikzlibrary{spy}
\usepackage{pgfplots}
  \pgfplotsset{width=7cm,compat=1.8}

\begin{document}
\maketitle
\begin{abstract}
 The Condat--V\~u algorithm is a widely used primal-dual method for optimizing composite objectives of three functions. Several algorithms for optimizing composite objectives of two functions are special cases of Condat--V\~u, including proximal gradient descent (PGD). It is well-known that PGD exhibits suboptimal performance, and a simple adjustment to PGD can accelerate its convergence rate from $\mathcal{O}(1/T)$ to $\mathcal{O}(1/T^2)$ on convex objectives, and this accelerated rate is optimal. In this work, we show that a simple adjustment to the Condat--V\~u algorithm allows it to recover accelerated PGD (APGD) as a special case, instead of PGD. We prove that this accelerated Condat--V\~u algorithm achieves optimal convergence rates and significantly outperforms the traditional Condat--V\~u algorithm in regimes where the Condat--V\~u algorithm approximates the dynamics of PGD. We demonstrate the effectiveness of our approach in various applications in machine learning and computational imaging.
\end{abstract}



\section{Introduction}

Let $\mathcal{X}$ and $\mathcal{Y}$ be vector spaces with inner-product $\langle \cdot , \cdot \rangle$ and induced norm $\|\cdot\|$. This work is concerned with solving problems of the form
\begin{equation}
    \label{eq:main}
    \min_{x \in \mathcal{X}} \quad F(x) \defeq f(A x) + g(x) + h(x),
\end{equation}
where $A : \mathcal{X} \to \mathcal{Y}$ is a linear operator that is bounded with respect to the standard operator norm, $\|\cdot\|_{\textnormal{op}}$, induced by $\|\cdot\|$; $f,g,h : \mathcal{X} \to \mathbb{R} \cup \{+\infty\}$ are convex, proper, lower semicontinuous functions; and $h$ has a Lipschitz continuous gradient with constant $L$. We further assume that $f$ and $g$ are ``simple'' functions, in the sense that their proximal operators (see \eqref{eq:prox}) are efficiently computable. This well-studied class of problems often arises in image processing and machine learning applications, including regularized regression problems such as LASSO and the elastic net \cite{lasso,elasticnet}, and image denoising and deblurring using total variation regularization and its variants \cite{cpacta} A particular difficulty of this class of problems is that even when the proximal operators of $f$ and $g$ are easy to compute, it is generally difficult to compute the proximal operator of $f \circ A$ and $f \circ A + g$, rendering many proximal methods inapplicable.

Instead of solving problem \eqref{eq:main} directly, it is common to solve the following saddle-point problem:
\begin{equation}
    \label{eq:sp}
    \min_{x \in \mathcal{X}} \max_{y \in \mathcal{Y}} \quad \langle A x, y \rangle - f^*(y) + g(x) + h(x).
\end{equation}
Here, $f^*$ is the convex conjugate of $f$, defined as $f^*(y) \defeq \sup_z \langle y,z \rangle - f(z)$. Under mild constraint qualifications \cite{combettepesquet}, the set of optimal primal variables in \eqref{eq:sp} is the solution set of \eqref{eq:main}.
We outline some popular algorithms for solving \eqref{eq:main}, or special cases of \eqref{eq:main}, below.

\paragraph{Accelerated proximal gradient descent (APGD) \cite{lincoup, fista,chambolle2015convergence}} Proximal gradient descent methods \cite{combettes2005signal, combettes2011proximal}, particularly the accelerated methods APGD minimises objectives that are the sum of a non-smooth function with a computable proximal operator and a smooth function with a Lipschitz continuous gradient. Assuming the proximal operator of $f \circ A$ cannot be computed efficiently, APGD is only applicable to \eqref{eq:main} when $f$ is smooth. In this case, we absorb $f \circ A$ into the smooth function $h$. APGD is a theoretically strong choice for solving this special case of \eqref{eq:main} because it achieves optimal convergence rates 
of $\mathcal{O}(L / T^2)$ (where $T$ is the number of iterations) without strong convexity and $\mathcal{O}( (1 + 1/\sqrt{\kappa_P})^{-T})$ (where $\kappa_P \defeq L / \mu_g$ is the condition number associated with this ``primal'' problem) when the objective is $\mu_g$-strongly convex \cite{lincoup, fista, chambolle2015convergence}.

\paragraph{Condat--V\~u \cite{CPrates,Condat2013,Vu2013condatvu}} The Condat--V\~u algorithm applies to \eqref{eq:sp} in its full generality. When $h \equiv 0$, the Condat--V\~u algorithm reduces to the Chambolle--Pock algorithm. However, when $f \circ A \equiv 0$, the Condat--V\~u algorithm does not recover APGD, but reduces to the non-accelerated variant, proximal gradient descent. This is reflected in the algorithm's performance. For large $\|A\|_{\textnormal{op}}$ and $L$, the Condat--V\~u algorithm's performance is strictly worse than APGD and the Chambolle--Pock algorithm, converging like $\mathcal{O}(\frac{L + \|A\|_{\textnormal{op}}}{T})$, $\mathcal{O}(\frac{L + \|A\|_{\textnormal{op}}^2}{\mu_g T^2})$, and $\mathcal{O}( (1 + \sqrt{\kappa_{PD}} )^{-T} + (1 + \kappa_P)^{-T} )$, in the general, strongly convex, and strongly convex and smooth settings, respectively \cite{CPrates}.

\paragraph{PDHG \cite{chambollepock}} The \emph{Primal-Dual Hybrid Gradient (PDHG)} algorithm, or the \emph{Chambolle--Pock} algorithm solves the saddle-point problem \eqref{eq:sp} and is only applicable when $h \equiv 0$. In this setting, it is also a theoretically strong choice, achieving optimal convergence rates of $\mathcal{O}(\|A\|_{\textnormal{op}}^2 / T)$ when both functions are non-smooth and non-strongly convex, $\mathcal{O}(\frac{\|A\|_{\textnormal{op}}^2}{\mu_g T^2})$ when $g$ is $\mu_g$-strongly convex, and $\mathcal{O}((1 + 1/\sqrt{\kappa_{PD}})^{-T})$ when $g$ is strongly convex and $f^*$ is $\mu_{f^*}$-strongly convex (where $\kappa_{PD} \defeq \|A\|_{\textnormal{op}}^2 / (\mu_g \mu_{f^*})$ is the condition number associated with the \emph{primal-dual} problem).\footnote{$f^*$ being $\mu_{f^*}$-strongly convex is equivalent to $f$ being differentiable with a $1/\mu_{f^*}$-Lipschitz continuous gradient, in which case the Lipschitz constant of $f \circ A$ is less than $\|A\|_{\textnormal{op}}^2 / \mu_{f^*}$. This makes the relationship between $\kappa_{PD}$ and $\kappa_P$ clear. Both condition numbers are of the form $L' / \mu'$ for relevant Lipschitz constant $L'$ and strong-convexity parameter $\mu'$. $\kappa_P$ is the condition number of $h + g$, and $\kappa_{PD}$ is the condition number of $f \circ A + g$.} Each of these rates come from the original work of Chambolle and Pock using parameter settings that essentially optimize the theoretical asymptotic convergence rates \cite{chambollepock}.

\paragraph{Accelerated primal-dual (APD) \cite{chen13} and extensions \cite{zhao21,zhao19}} Recognising the suboptimal performance of the Condat--V\~u algorithm, Chen \emph{et al.}\! developed the APD algorithm, which incorporates momentum into the Condat--V\~u algorithm with $g \equiv 0$. APD achieves the optimal convergence rate of $\mathcal{O}(L / T^2 + \|A\|_{\textnormal{op}} / T)$ in the general setting, but its applicability is limited because it requires $g \equiv 0$. Due to this restriction, the algorithm is not known to admit accelerated convergence rates when strong convexity or additional smoothness is present. Zhao \emph{et al.}\! \cite{zhao19} propose a similar method, the \emph{optimal primal-dual hybrid gradient (OTPDHG)} algorithm, that uses stochastic gradient estimates and achieves the same convergence rates as APD even with non-trivial $g$, but the authors do not analyze the case where strong convexity or additional smoothness is present. In work concurrent to this manuscript, Zhao generalises OTPDHG to solve a larger class of saddle-point problems. Zhao proves that the resulting algorithm achieves optimal convergence rates on the bilinear saddle-point problem in \eqref{eq:sp} when all functions are convex, and when $g$ is additionally strongly convex. Zhao's approach requires a restart scheme to achieve an optimal convergence rate of $\mathcal{O}(\tfrac{L}{\mu_g (1 + \sqrt{\kappa_P})^T} + \tfrac{\|A\|_{\textnormal{op}}^2}{\mu_g T^2})$ in the strongly convex setting, and Zhao does not consider the case where $f$ is smooth\footnote{One may argue that for this regime we could use directly the restarted FISTA/APGD algorithm (considering $f+h$ as the smooth part) to achieve this rate in practice hence we do not need to study this case. We dispute against this argument, since there are cases where $1/\mu_{f^\star}$ is huge (which means the objective has only very limited smoothness), and for this case the primal methods FISTA/APGD will need to take a tiny step-size all the way to ensure convergence, while we can use a two-stage scheme for the step-sizes of ACV, switching from the ``SC" rate to the ``SC and Smooth" rate, enabling large step-sizes in early iterations and hence faster convergence. Hence we insist that the accelerated Condat-V\~u scheme still has an important role to play in such a regime.}.

\paragraph{Our contribution: Accelerated Condat--V\~u (ACV)} This work introduces the Accelerated Condat--V\~u algorithm (ACV), which incorporates momentum into the Condat--V\~u algorithm. ACV achieves \emph{optimal convergence rates}, in a sense we specify below.
\begin{itemize}
    \item In the general case, ACV converges like $\mathcal{O}(L / T^2 + \|A\|_{\textnormal{op}} / T)$. There exists a smooth, convex function $h$ such that no first-order method can minimize $h$ at a rate exceeding $\mathcal{O}(L/T^2)$ \cite{nest2004}. Furthermore, there exists $b \in \mathcal{Y}$ and $A$ such that no first-order method can solve the problem $\min_{x \in \mathcal{X}} \max_{y \in \mathcal{Y}} \langle A x - b, y \rangle$ at a rate exceeding $\mathcal{O}(\|A\|_{\textnormal{op}} / T)$ \cite{chen13,ouyang_comp}.
    \item When $g$ (or $h$)\footnote{Strong convexity can be ``transfered'' between $g$ and $h$ by adding and subtracting $\tfrac{\mu_g}{2} \|x\|^2$ without complicating any computations.} is $\mu_g$-strongly convex, ACV converges like $\mathcal{O}(\tfrac{L}{\mu_g (1 + \sqrt{\kappa_P})^T} + \tfrac{\|A\|_{\textnormal{op}}^2}{\mu_g T^2} )$. There exists a smooth, strongly convex function that no first-order algorithm can minimize at a rate faster than $\mathcal{O} ((1 + \sqrt{\kappa_P})^{-T})$ \cite{nemiyudin}, and there exists a $b \in \mathcal{Y}$ such that no first-order algorithm can solve $\min_{x \in \mathcal{X}} \max_{y \in \mathcal{Y}} \langle K x - b, y \rangle$ at a rate exceeding $\mathcal{O}(\tfrac{\|A\|_{\textnormal{op}}^2}{\mu_g T^2})$ \cite{ouyang_comp, zhao21}.
    \item When $g$ (or $h$) is $\mu_g$-strongly convex and $f$ is $1/\mu_{f^*}$-smooth, then ACV achieves a convergence rate of $\mathcal{O}((1 + \kappa_{PD}^{-1/2} + \kappa_P^{-1/2})^{-T})$. Given that the condition number of the objective $f(Ax) + h(x)$ is $\kappa_{PD} + \kappa_P$, it is clear from the previous point that no first-order algorithm can improve on this convergence rate in general.
\end{itemize}
As the ratio $L / \|A\|_{\textnormal{op}}$ increases, ACV interpolates between the Chambolle--Pock algorithm and APGD, reducing to the accelerated variant of proximal gradient descent when $f \circ A = 0$ and achieving optimal convergence rates. Because the Condat--V\~u algorithm interpolates between the Chambolle--Pock algorithm and non-accelerated proximal gradient descent, ACV significantly outperforms the Condat--V\~u algorithm in theory and practice. This is especially true in the regime $L \gg \|A\|_{\textnormal{op}}$, where ACV behaves like APGD, while the Condat--V\~u algorithm behaves like the non-accelerated variant, PGD. Algorithm \ref{alg:framework} shows the general update rule for ACV. It reduces to the Condat--V\~u algorithm when $\alpha_k = 1$. While the general structure of ACV is similar to the algorithm proposed in \cite{zhao19}, different parameter settings, warm-up techniques, and analytical tools are required to prove the state-of-the-art convergence rates presented in this work and achieve them in practice. We show that ACV converges for a wide range of parameters that include the parameters suggested in \cite{zhao19}. We also show that ACV achieves optimal convergence rates when $f$ is smooth, which is a case not considered in other works \cite{zhao19,zhao21}, and ACV achieves optimal rates in the strongly convex setting without requiring a restart scheme as in \cite{zhao21}.

Table \ref{tab:rates1} summarises these algorithms and their convergence rates. Developing an algorithm that achieves optimal convergence rates over the entire class of problems represented by \eqref{eq:main} is not solely for theoretical interest. In many applications of the Condat--V\~u algorithm, $L \gg \|A\|_{\textnormal{op}}$, forcing the Condat--V\~u algorithm to operate in the regime where its theoretical performance is furthest from optimal. Popular examples include regression problems with composite regularizers, often formulated as
\begin{equation}
    \min_{x \in \mathcal{X}} \quad \frac{1}{2} \|W x - y\|^2 + \lambda_1 \mathcal{R}_1(x) + \lambda_2 \mathcal{R}_2(x).
\end{equation}
%
%
Here, $y$ is a vector of labels, $W$ is a data matrix of features, and $\mathcal{R}_1, \mathcal{R}_2$ are convex, possibly non-smooth regularizers with tuning parameters $\lambda_1, \lambda_2$. The elastic net regularizer for feature selection falls into this class of problems with $\mathcal{R}_1(x) = \|x\|_1$ and $\mathcal{R}_2(x) = \tfrac{1}{2}\|x\|^2$. Another example is LASSO with a non-negative constraint, where $\mathcal{R}_2$ is the indicator function of the non-negative orthant. For these problems, the natural choice $h(x) = \tfrac{1}{2} \|W x - y\|^2$ leads to a very large $L = \|W\|_{\textnormal{op}}^2$, making the Condat--V\~u algorithm extremely inefficient, as we demonstrate in our numerical experiments in Section \ref{sec:ex}


\begin{table}[!t]
\begin{minipage}{\textwidth}
\footnotesize
\centering
\begin{tabular}{|p{1.7cm}| C{1.55cm} | C{1.6cm} | C{1.5cm}| C{3.1cm} | C{3.1cm} | C{2.2cm} |}
\hline
& Restriction & Easy-to-tune? & Convex & Strongly Convex (SC) & SC and Smooth \\
\hline
APGD \cite{lincoup} & $f \circ A \equiv 0$ & {Yes} & $ \frac{L}{T^2} $ & $ \frac{L}{\mu_g (1 + \kappa_P^{-1/2} )^{T}} $ & --- \\
\hline
PDHG \cite{chambollepock} & $h \equiv 0$ & {Yes} &  $ \frac{\|A\|_{\textnormal{op}}}{T} $ & $ \frac{\|A\|_{\textnormal{op}}^2}{\mu_g T^2} $ & $ (1 + \kappa_{PD}^{-1/2} )^{-T} $ \\
\hline
Zhao \emph{et al.}\! \cite{zhao19} & None & {No} & $ \frac{L}{T^2} + \frac{\|A\|_{\textnormal{op}}}{T} $ & --- & --- \\
\hline
Zhao \! \cite{zhao21} & None & {No} & $ \frac{L}{T^2} + \frac{\|A\|_{\textnormal{op}}}{T} $ & $ \frac{L}{\mu_g (1 + \kappa_P^{-1/2})^{T}} + \frac{\|A\|_{\textnormal{op}}^2}{\mu_g T^2} $ & --- \\
\hline
Condat--V\~u \cite{CPrates,Condat2013,Vu2013condatvu} & None & {Yes} & $ \frac{L + \|A\|_{\textnormal{op}}}{T} $ & $ \frac{L + \|A\|_{\textnormal{op}}^2}{\mu_g T^2} $ & $ (1 + \kappa_{PD}^{-1/2} + \kappa_P^{-1} )^{-T} $ \\
\hline
\textbf{ACV (Ours)} & None & {Yes} & $ \frac{L}{T^2} + \frac{\|A\|_{\textnormal{op}}}{T} $ & $ \frac{L}{\mu_g (1 + \kappa_P^{-1/2})^{T}} + \frac{\|A\|_{\textnormal{op}}^2}{\mu_g T^2} $ & $ (1 + \kappa_{PD}^{-1/2} + \kappa_P^{-1/2} )^{-T} $ \\
\hline
\end{tabular}
\caption{%
Best-known convergence rates for solving problems of the form \eqref{eq:main} up to constant factors. ``Convex'' refers to the general setting of \eqref{eq:main}, ``Strongly convex'' implies $g$ or $h$ is $\mu_g$-strongly convex, and ``smooth'' implies $f$ has a $1/\mu_{f^*}$-Lipschitz continuous gradient. $\kappa_{PD} = \tfrac{\|A\|^2_{\textnormal{op}}}{\mu_{f^*} \mu_g}$ and $\kappa_P = L/\mu_g$ are condition numbers. 
These rates are with respect to different quantities; please see the associated works for further details. The column ``Practicality" indicates the difficulty of finding suitable choices for the algorithmic parameters, as well as the range of applications. All convergence rates are displayed assuming parameter settings that approximately optimize theoretical convergence rates. ACV interpolates between the optimal rates of APGD and PDHG.}
\label{tab:rates1}
\end{minipage}
\end{table}


\begin{center}
\begin{minipage}{0.975\linewidth}
\begin{algorithm}[H]
\caption{Accelerated Condat--V\~u Framework: {  For $\primal_0, \momentumtwo_0 \in \mathcal{X}$, $\dualit_0 \in \mathcal{Y}$, and $\gamma_k, \tau_k, \alpha_k, \theta_k \in \mathbb{R}$ for all $k \in \{0,1,\cdots,T_{\max} - 1\}$, define ACV$(\primal_0,\dualit_0,\momentumtwo_0, \{\gamma_k\},\{\taustep_k\},\{\alpha_k\},$ $\{\theta_k\},T_{\max})$ as the following:}}
\label{alg:framework}
  \begin{algorithmic}[1]
  \State \textbf{For} $k = 0,1,\cdots, T_{\max} - 1$:
  \State $\momentumone_{k+1} \leftarrow \alpha_k \primal_k + (1 - \alpha_k) \momentumtwo_k$
    \State $\dualit_{k+1} \leftarrow \operatorname{prox}_{\gamma_k f^*} ( \dualit_k + \gamma_k A ( \primal_k + \theta_k (\primal_k - \primal_{k-1}) ) )$
    %
    %
    \State $\primal_{k+1} \leftarrow \operatorname{prox}_{\taustep_k g} ( \primal_k - \taustep_k \nabla h(\momentumone_{k+1}) - \taustep_k A^* \dualit_{k+1})$
    \State $\momentumtwo_{k+1} \leftarrow \alpha_k \primal_{k+1} + (1 - \alpha_k) \momentumtwo_k$.
    \State \textbf{end for}
  \end{algorithmic}
\end{algorithm}
\end{minipage}
\end{center}

\section{Related Work}

In addition to the works mentioned in the previous section, there are several primal-dual algorithms that apply to \eqref{eq:sp} or its special cases. Some of these algorithms offer minor improvements over the convergence rates of the Condat--V\~u algorithm, but only up to a constant factor, so we use the rates of the Condat--V\~u algorithm to represent all of them in Table \ref{tab:rates1}.

The Davis--Yin splitting algorithm is a method for finding the zero of the sum of three maximally monotone operators, generalizing popular two-operator splitting methods such as forward-backward splitting.
Many algorithms for minimizing the sum of three functions are special cases of the Davis--Yin splitting algorithm, including PD3O \cite{pd3o}, PDDY \cite{Salim2020condatvu}, and Condat--V\~u \cite{Condat2013,Vu2013condatvu}. PD3O and PDDY exhibit essentially the same convergence rates as Condat--V\~u.

The asymmetric forward-backward-adjoint (AFBA) algorithm, { which is intrinsically the same as Davis--Yin splitting as pointed out by Condat et al \cite{condat2022randprox}}, finds the zero of the sum of three maximally monotone operators, and is hence applicable to solve problems of the form \eqref{eq:main}.
%
%
Primal-dual fixed-point (PDFP) \cite{pdfp} is another algorithm to solve problems of the form \eqref{eq:sp} in its full generality. Unlike the other algorithms we have mentioned, PDFP requires two evaluations of the proximal operator of $g$ per iteration, which can make it slightly more expensive in practice.{ Moreover, recently there are a serge of works on acceleration of fixed-point iterations, starting from the work of Kim \cite{kim2021accelerated} on accelerated proximal-point schemes and the work of Yoon et al \cite{yoon2021accelerated} on extra anchored gradient achieving $O(1/K^2)$ rate on the squared gradient norm, to the work of  Bot et al \cite{bot2022fast} on the acceleration of Krasnosel'skii-Mann iteration and Park et al \cite{park2022exact} for more general fixed-point acceleration schemes.}

Many other popular primal-dual algorithms are special cases of PDFP, AFBA, PD3O, and the Condat--V\~u algorithm \cite{condat2023proximal}. The preconditioned alternating projection algorithm (PAPA) \cite{papa} is equivalent to PDFP when $g$ is the indicator function of a convex set. If $g \equiv 0$, then PDFP, AFBA, and PD3O all reduce to the proximal alternating predictor-corrector (PAPC) algorithm \cite{papc}. PD3O and Condat--V\~u reduce to the Chambolle--Pock algorithm when $h \equiv 0$. Figure \ref{fig:alg} summarises these algorithms and their relationships in Appendix \ref{app:fig}.

As Figure \ref{fig:alg} shows, most of these algorithms reduce to (non-accelerated) proximal gradient descent when $f \circ A \equiv 0$, preventing them from achieving optimal convergence rates. The only exception is PDHG, which achieves optimal convergence rates. 
Although the focus of this manuscript is accelerating the Condat--V\~u algorithm, it is likely that similar techniques can be used to accelerate PDDY and PD3O as well, leading to accelerated variants of each of the algorithms in the tree of Figure \ref{fig:alg}. 

Recent work by Salim \emph{et al.}\! takes a different approach to accelerating primal-dual algorithms. This work develops a framework for incorporating stochastic gradient estimators, such as the SVRG estimator \cite{svrg}, into several primal-dual algorithms, including Condat--V\~u \cite{Salim2020condatvu} and PDDY. Because these stochastic estimates of the gradient are cheaper to compute than the full gradient, the proposed algorithms have a smaller overall complexity cost than their full-gradient analogues. { In certain applications, one could also consider randomizing the proximal operator $\mathrm{prox}_f$ instead of the gradients $\nabla h$ as pointed out by Condat and Rickt\'arik \cite{condat2022randprox}.} None of these algorithms incorporate acceleration techniques, so although they have a reduced per-iteration complexity, they exhibit the same convergence rates as their full-gradient counterparts.

\section{Preliminaries}

In this section, we provide the background necessary for our analysis of Algorithm \ref{alg:framework}.

The \emph{Bregman divergence} \cite{bregman} associated with a function $g$ is defined as
\begin{equation}
    D^\xi_g(u,x) \defeq g(u) - g(x) - \langle \xi, u - x \rangle,
\end{equation}
where $\xi \in \partial g(x)$ and $\partial$ is the subdifferential operator. The function $g$ is convex if and only if $D^\xi(u,x) \ge 0$ for all $x,u \in \mathcal{X}$. If $g$ is strictly convex, then $D^\xi_g(u,x) = 0$ implies $x = y$.

Some of our convergence guarantees are stated as bounds on the \emph{(partial) primal-dual gap}. Define the Lagrangian
\begin{equation}
    \mathcal{L}(x,y) \defeq \langle A x, y \rangle - f^*(y) + g(x) + h(x)
\end{equation}
%
The primal-dual gap at $(x,y) \in \mathcal{X} \times \mathcal{Y}$ is defined as
%
%
\begin{equation}
    PD_{\mathcal{B}_1 \times \mathcal{B}_2}(x,y) \defeq \max_{y' \in \mathcal{B}_2} \mathcal{L}(x,y') - \min_{x' \in \mathcal{B}_1} \mathcal{L}(x',y),
\end{equation}
where we take the set $\mathcal{B}_1 \times \mathcal{B}_2 \subset \mathcal{X} \times \mathcal{Y}$ to be a compact set containing a saddle-point solution to \eqref{eq:sp}. The primal-dual gap is non-negative and equal to zero only at saddle-points of \eqref{eq:sp}.


A function $g$ is $\mu_g$-\emph{strongly convex} if and only if
\begin{equation}
\label{eq:defstrong}
    \frac{\mu_g}{2} \|u - x\|^2 \le D^\xi_g(u,x), 
\end{equation}
for all $x,u \in \mathcal{X}$. If $g$ is $\mu_g$-strongly convex, then its convex conjugate $g^*$ has a $1/\mu_g$-Lipschitz continuous gradient \cite{cpacta}. If a function $h$ has an $L$-Lipschitz continuous gradient, then \cite[Thm. 2.1.5]{nest2004}
\begin{equation}
    D^{\zeta}_h(u,x) \le \frac{L}{2} \|x - u\|^2.
\end{equation}

The \emph{proximal operator} is defined as
\begin{equation}
\label{eq:prox}
\prox_{g} (u) = \argmin_{x} \quad \frac{1}{2} \|x-u\|^2 + g(x).
\end{equation}
The proximal operator is also defined implicitly as $u - \prox_g(u) \in \partial g(\prox_g(u))$, which follows from the optimality conditions of the problem in \eqref{eq:prox}.

The following two lemmas establish a fundamental inequality on the duality gap at iterates of Algorithm \ref{alg:framework}. In later sections, we adapt this inequality to prove convergence rates in the general setting, and in the presence of strong convexity and additional smoothness in the objective.

\begin{lemma}
\label{lem:cvinit}
Suppose $g$ is $\mu_g$-strongly convex and $f^*$ is $\mu_{f^*}$-strongly convex with $\mu_g, \mu_{f^*}$ $\ge 0$. The {  iterates of Algorithm \ref{alg:framework} satisfy the following inequality} for any $(x,y) \in \mathcal{X} \times \mathcal{Y}$:
\begin{equation}
\begin{aligned}
    & \gamma_k \big( f^*(\dualit_{k+1}) - f^*(y) + g(\primal_{k+1}) - g(x) + \langle A \primal_{k+1}, y \rangle - \langle A x, \dualit_{k+1} \rangle \big) \\
    \le & \frac{\gamma_k}{2 \taustep_k} (\|\primal_k - x\|^2 - (1+\mu_g \taustep_k) \|\primal_{k+1} - x\|^2 - \|\primal_{k+1} - \primal_k\|^2) + \frac{1}{2} \|\dualit_k - y\|^2 \\
    & \quad - \frac{1 + \mu_{f^*} \gamma_k}{2} \|\dualit_{k+1} - y\|^2 + \frac{\theta_k^2 \gamma_k^2 \|A\|_{\textnormal{op}}^2}{2} \|\primal_k - \primal_{k-1}\|^2 - \gamma_k \langle \nabla h(\momentumone_{k+1}), \primal_{k+1} - x \rangle \\
    & \quad + \gamma_k \theta_k \langle A (\primal_k - \primal_{k-1}), \dualit_k - y \rangle - \gamma_k \langle A (\primal_{k+1} - \primal_k), \dualit_{k+1} - y \rangle
\end{aligned}
\end{equation}
\end{lemma}

\begin{proof}
By the definition of $\primal_{k+1}$ and the definition of the proximal operator, $1/\taustep_k (\primal_k - \taustep_k \nabla h(\momentumone_{k+1}) - \taustep_k A^* \dualit_{k+1} - \primal_{k+1}) \in \partial g(\primal_{k+1})$. Combining this fact with \eqref{eq:defstrong},
\begin{equation}
\begin{aligned}
    & \gamma_k \big( g(\primal_{k+1}) - g(x) \big) \\
    \le & \frac{\gamma_k}{\taustep_k} \langle \primal_k - \taustep_k \nabla h(\momentumone_{k+1}) - \taustep_k A^* \dualit_{k+1} - \primal_{k+1}, \primal_{k+1} - x \rangle - \frac{\mu_g \gamma_k}{2} \|\primal_{k+1} - x\|^2 \\
    =& \frac{\gamma_k}{\taustep_k} \langle \primal_k - \primal_{k+1}, \primal_{k+1} - x \rangle - \gamma_k \langle \nabla h(\momentumone_{k+1}), \primal_{k+1} - x \rangle \\
    & \quad - \gamma_k \langle A \primal_{k+1}, \dualit_{k+1} \rangle + \gamma_k \langle A x, \dualit_{k+1} \rangle - \frac{\mu_g \gamma_k}{2} \|\primal_{k+1} - x\|^2 \\
    =& \frac{\gamma_k}{\taustep_k} \langle \primal_k - \primal_{k+1}, \primal_{k+1} - x \rangle - \gamma_k \langle \nabla h(\momentumone_{k+1}), \primal_{k+1} - x \rangle \\
    & \quad - \gamma_k \langle A \primal_{k+1}, \dualit_{k+1} - y \rangle - \gamma_k \langle A \primal_{k+1}, y \rangle + \gamma_k \langle A x, \dualit_{k+1} \rangle \\
    &\quad - \frac{\mu_g \gamma_k}{2} \|\primal_{k+1} - x\|^2.
\end{aligned}
\end{equation}
Similarly, $1/\gamma_k (\dualit_k - \dualit_{k+1} + \gamma_k A(\primal_k + \theta_k (\primal_k - \primal_{k-1}) ) ) \in \partial f^*(\dualit_{k+1})$, allowing us to say
\begin{equation}
\begin{aligned}
    & \gamma_k \big( f^*(\dualit_{k+1}) - f^*(y) \big) \\
    \le & \langle \dualit_k + \gamma_k A (\primal_k + \theta_k (\primal_k - \primal_{k-1}) ) - \dualit_{k+1}, \dualit_{k+1} - y \rangle - \frac{\mu_{f^*} \gamma_k}{2} \|\dualit_{k+1} - y\|^2 \\
    = & \langle \dualit_k - \dualit_{k+1}, \dualit_{k+1} - y \rangle + \gamma_k \langle A (\primal_k + \theta_k (\primal_k - \primal_{k-1}) ), \dualit_{k+1} - y \rangle \\
    & \quad - \frac{\mu_{f^*} \gamma_k}{2} \|\dualit_{k+1} - y\|^2.
\end{aligned}
\end{equation}
Adding these inequalities, we have
\begin{equation}
\begin{aligned}
    & \gamma_k \big( f^*(\dualit_{k+1}) - f^*(y) + g(\primal_{k+1}) - g(x) + \langle A \primal_{k+1}, y \rangle - \langle A x, \dualit_{k+1} \rangle \big) \\
    \le & \frac{\gamma_k}{\taustep_k} \langle \primal_k - \primal_{k+1}, \primal_{k+1} - x \rangle + \langle \dualit_k - \dualit_{k+1}, \dualit_{k+1} - y \rangle \\
    & \quad + \gamma_k \langle A (\primal_k + \theta_k (\primal_k - \primal_{k-1}) ), \dualit_{k+1} - y \rangle - \gamma_k \langle \nabla h(\momentumone_{k+1}), \primal_{k+1} - x \rangle \\
    &\quad - \gamma_k \langle A \primal_{k+1}, \dualit_{k+1} - y \rangle - \frac{\mu_g \gamma_k}{2} \|\primal_{k+1} - x\|^2 - \frac{\mu_{f^*} \gamma_k}{2} \|\dualit_{k+1} - y\|^2.
\end{aligned}
\end{equation}
We simplify the first two terms on the right of this inequality using the relationship
\begin{equation}
\begin{aligned}
    \langle a - b, b - c \rangle = \frac{1}{2} \|a - c\|^2 - \frac{1}{2} \|b - c\|^2 - \frac{1}{2} \|a - b\|^2.
\end{aligned}
\end{equation}
We also bound the third and fifth terms using Young's inequality.
\begin{equation}
\begin{aligned}
    & \gamma_k \langle A (\primal_k + \theta_k (\primal_k - \primal_{k-1}) ), \dualit_{k+1} - y \rangle - \gamma_k \langle A \primal_{k+1}, \dualit_{k+1} - y \rangle \\
    = & \gamma_k \theta_k \langle A (\primal_k - \primal_{k-1}), \dualit_{k+1} - y \rangle - \gamma_k \langle A (\primal_{k+1} - \primal_k), \dualit_{k+1} - y \rangle \\
    = & \gamma_k \theta_k \langle A (\primal_k - \primal_{k-1}), \dualit_k - y \rangle - \gamma_k \langle A (\primal_{k+1} - \primal_k), \dualit_{k+1} - y \rangle \\
    & \quad + \gamma_k \theta_k \langle A (\primal_k - \primal_{k-1}), \dualit_{k+1} - \dualit_k \rangle \\
    \le & \gamma_k \theta_k \langle A (\primal_k - \primal_{k-1}), \dualit_k - y \rangle - \gamma_k \langle A (\primal_{k+1} - \primal_k), \dualit_{k+1} - y \rangle \\
    & \quad + \frac{\theta_k \gamma_k^2 \epsilon \|A\|_{\textnormal{op}}^2}{2} \|\primal_k - \primal_{k-1}\|^2 + \frac{\theta_k}{2 \epsilon} \|\dualit_{k+1} - \dualit_k\|^2,
\end{aligned}
\end{equation}
for any $\epsilon > 0$. This produces the inequality
\begin{equation}
\begin{aligned}
    & \gamma_k \big( f^*(\dualit_{k+1}) - f^*(y) +  g(\primal_{k+1}) - g(x) + \langle A \primal_{k+1}, y \rangle - \langle A x, \dualit_{k+1} \rangle \big) \\
    \le & \frac{\gamma_k}{2 \taustep_k} (\|\primal_k - x\|^2 - (1+\mu_g \taustep_k) \|\primal_{k+1} - x\|^2 - \|\primal_{k+1} - \primal_k\|^2) + \frac{1}{2} \|\dualit_k - y\|^2 \\
    & - \frac{1 + \mu_{f^*} \gamma_k}{2} \|\dualit_{k+1} - y\|^2 - \left( \frac{1}{2} - \frac{\theta_k}{2 \epsilon} \right) \|\dualit_{k+1} - \dualit_k\|^2 + \frac{\theta_k \gamma_k^2 \epsilon \|A\|_{\textnormal{op}}^2}{2} \|\primal_k - \primal_{k-1}\|^2 \\
    & - \gamma_k \langle \nabla h(\momentumone_{k+1}), \primal_{k+1} - x \rangle + \gamma_k \theta_k \langle A (\primal_k - \primal_{k-1}), \dualit_k - y \rangle \\
    & - \gamma_k \langle A (\primal_{k+1} - \primal_k), \dualit_{k+1} - y \rangle
\end{aligned}
\end{equation}
The choice $\epsilon = \theta_k$ minimises the coefficient of $\|\primal_k - \primal_{k-1}\|^2$ while ensuring that the coefficient of $\|\dualit_{k+1} - \dualit_k\|^2$ is non-positive. This completes the proof.
\end{proof}

The following lemma introduces the smooth function, $h$, into the inequality of Lemma \ref{lem:cvinit} and bounds the duality gap.

\begin{lemma}[One-Iteration Progress]
\label{lem:cvtele}
    Let $w_{k+1} = \alpha_k \dualit_{k+1} + (1-\alpha_k) \momentumtwo_k$. The {  iterates of Algorithm \ref{alg:framework} satisfy the following inequality} for all $(x,y) \in \mathcal{X} \times \mathcal{Y}$:
    \begin{equation}
\begin{aligned}
    \label{eq:telelast}
        & \frac{\gamma_k}{\alpha_k} ( \mathcal{L}(\momentumtwo_{k+1}, y) - \mathcal{L}(x, w_{k+1}) ) \\
        \le & \frac{\gamma_k (1 - \alpha_k)}{\alpha_k} ( \mathcal{L}(\momentumtwo_k, y) - \mathcal{L}(x, w_k) ) + \frac{\gamma_k}{2 \taustep_k} \left( L \alpha_k \taustep_k - 1 \right) \|\primal_{k+1} - \primal_k\|^2 \\
        & \quad + \frac{\theta_k^2 \gamma_k^2 \|A\|_{\textnormal{op}}^2}{2} \|\primal_k - \primal_{k-1}\|^2 + \frac{\gamma_k}{2 \taustep_k} (\|\primal_k - x\|^2 - (1 + \mu_g \taustep_k) \|\primal_{k+1} - x\|^2) \\
        & \quad + \frac{1}{2} \|\dualit_k - y\|^2 - \frac{1 + \mu_{f^*} \gamma_k}{2} \|\dualit_{k+1} - y\|^2 \\
        & \quad + \gamma_k \theta_k \langle A (\primal_k - \primal_{k-1}), \dualit_k - y \rangle - \gamma_k \langle A (\primal_{k+1} - \primal_k), \dualit_{k+1} - y \rangle.
    \end{aligned}
\end{equation}
\end{lemma}

\begin{proof}
We begin by bounding the suboptimality of $h$. Using the convexity of $h$,
\begin{equation}
\begin{aligned}
    & \gamma_k (h(\momentumone_{k+1}) - h(x)) \\
    & \le \gamma_k \langle \nabla h(\momentumone_{k+1}), \momentumone_{k+1} - x \rangle \\
    & = \gamma_k \big( \langle \nabla h(\momentumone_{k+1}), \momentumone_{k+1} - \primal_k \rangle + \langle \nabla h(\momentumone_{k+1}), \primal_k - \primal_{k+1} \rangle + \langle \nabla h(\momentumone_{k+1}), \primal_{k+1} - x \rangle \big).
\end{aligned}
\end{equation}
Using the definitions $\momentumone_{k+1} - \primal_k = \frac{1 - \alpha_k}{\alpha_k} (\momentumtwo_k - \momentumone_{k+1})$ and $\primal_k - \primal_{k+1} = 1 / \alpha_k ( \momentumone_{k+1} - \momentumtwo_{k+1} )$ from lines 2 and 5 in Algorithm \ref{alg:framework}, this becomes
\begin{equation}
\begin{aligned}
    \gamma_k (h(\momentumone_{k+1}) - h(x)) & \le \frac{\gamma_k (1 - \alpha_k)}{\alpha_k} \langle \nabla h(\momentumone_{k+1}), \momentumtwo_k - \momentumone_{k+1} \rangle \\
    & \quad + \frac{\gamma_k}{\alpha_k} \langle \nabla h(\momentumone_{k+1}), \momentumone_{k+1} - \momentumtwo_{k+1} \rangle + \gamma_k \langle \nabla h(\momentumone_{k+1}), \primal_{k+1} - x \rangle \\ 
%
%
    & \le \frac{\gamma_k (1 - \alpha_k)}{\alpha_k} ( h(\momentumtwo_k) - h(\momentumone_{k+1}) ) + \frac{\gamma_k}{\alpha_k} ( h(\momentumone_{k+1}) - h(\momentumtwo_{k+1}) ) \\
    & \quad + \frac{L \gamma_k}{2 \alpha_k} \|\momentumone_{k+1} - \momentumtwo_{k+1}\|^2 + {  \gamma_k} \langle \nabla h(\momentumone_{k+1}), \primal_{k+1} - x \rangle,
\end{aligned}
\end{equation}
where the final inequality follows from the convexity and smoothness of $h$. From here on, we write $\|\momentumone_{k+1} - \momentumtwo_{k+1}\|^2$ as $\alpha_k^2 \|\primal_{k+1} - \primal_k\|^2$. We bound the final term using Lemma \ref{lem:cvinit}.
\begin{equation}
\begin{aligned}
    & \gamma_k (h(\momentumone_{k+1}) - h(x) + f^*(\dualit_{k+1}) - f^*(y) + g(\primal_{k+1}) - g(x) \\
    & \quad + \langle A \primal_{k+1}, y \rangle - \langle A x, \dualit_{k+1} \rangle) \\
    \le & \frac{\gamma_k (1 - \alpha_k)}{\alpha_k} ( h(\momentumtwo_k) - h(\momentumone_{k+1}) ) + \frac{\gamma_k}{\alpha_k} ( h(\momentumone_{k+1}) - h(\momentumtwo_{k+1}) ) \\
    & \quad + \frac{\gamma_k}{2 \taustep_k} (\|\primal_k - x\|^2 - (1+\mu_g \taustep_k) \|\primal_{k+1} - x\|^2 + (L \alpha_k \taustep_k - 1)\|\primal_{k+1} - \primal_k\|^2) \\
    & \quad + \frac{1}{2} \|\dualit_k - y\|^2 - \frac{1 + \mu_{f^*} \gamma_k}{2} \|\dualit_{k+1} - y\|^2 + \frac{\theta_k^2 \gamma_k^2 \|A\|_{\textnormal{op}}^2}{2} \|\primal_k - \primal_{k-1}\|^2 \\
    & \quad + \gamma_k \theta_k \langle A (\primal_k - \primal_{k-1}), \dualit_k - y \rangle - \gamma_k \langle A (\primal_{k+1} - \primal_k), \dualit_{k+1} - y \rangle.
\end{aligned}
\end{equation}
After simplifications, this is equivalent to
\begin{equation}
\label{eq:pd}
\begin{aligned}
    & \frac{\gamma_k}{\alpha_k} \big( h(\momentumtwo_{k+1}) - h(x) \big) + \gamma_k \big( g(\primal_{k+1}) + f^*(\dualit_{k+1}) - g(x) - f^*(y) \\
    & \quad + \langle A \primal_{k+1}, y \rangle - \langle A x, \dualit_{k+1} \rangle \big) \\
    \le & \frac{\gamma_k (1 - \alpha_k)}{\alpha_k} ( h(\momentumtwo_k) - h(x) ) + \frac{\gamma_k}{2 \taustep_k} \left(L \alpha_k \taustep_k - 1 \right) \|\primal_{k+1} - \primal_k\|^2 \\
    & \quad + \frac{\theta_k^2 \gamma_k^2 \|A\|_{\textnormal{op}}^2}{2} \|\primal_k - \primal_{k-1}\|^2 + \frac{\gamma_k}{2 \taustep_k} (\|\primal_k - x\|^2 - (1 + \mu_g \taustep_k) \|\primal_{k+1} - x\|^2) \\
    & \quad + \frac{1}{2} \|\dualit_k - y\|^2 - \frac{1 + \mu_{f^*} \gamma_k}{2} \|\dualit_{k+1} - y\|^2 + \gamma_k \theta_k \langle A (\primal_k - \primal_{k-1}), \dualit_k - y \rangle \\
    & \quad - \gamma_k \langle A (\primal_{k+1} - \primal_k), \dualit_{k+1} - y \rangle.
\end{aligned}
\end{equation}
Letting $w_{k+1} = \alpha_k \dualit_{k+1} + (1 - \alpha_k) w_k$ (with $w_0 = \dualit_0$), we use the convexity of $g$ and $f^*$ to say $g(\momentumtwo_{k+1}) \le \alpha_k g(\primal_{k+1}) + (1-\alpha_k) g(\momentumtwo_k)$ and $f^*(w_{k+1}) \le \alpha_k f^*(\dualit_{k+1}) + (1-\alpha_k) f^*(w_k)$. Therefore, 
%
\begin{equation}
\begin{aligned}
    & -\gamma_k \big( g(\primal_{k+1}) + f^*(\dualit_{k+1}) - g(x) - f^*(y) + \langle A \primal_{k+1}, y \rangle - \langle A x, \dualit_{k+1} \rangle \big) \\
    \le & -\frac{\gamma_k}{\alpha_k} \big( g(\momentumtwo_{k+1}) + f^*(w_{k+1}) - g(x) - f^*(y) + \langle A \momentumtwo_{k+1}, y \rangle - \langle A x, w_{k+1} \rangle \big) \\
    & \quad + \frac{\gamma_k (1 - \alpha_k)}{\alpha_k} \big( g(\momentumtwo_k) + f^*(w_k) - g(x) - f^*(y) + \langle A \momentumtwo_k, y \rangle - \langle A x, w_k \rangle \big).
\end{aligned}
\end{equation}
Combining this inequality with \eqref{eq:pd} completes the proof. 
\end{proof}

With the fundamental inequality \eqref{eq:telelast}, we are now prepared to prove convergence rates for the accelerated Condat--V\~u algorithm.

\section{Convergence analysis}

We prove convergence guarantees for Algorithm \ref{alg:framework} in three different settings. The first is the general setting, where we provide a range of parameters that guarantees the convergence of Algorithm \ref{alg:framework} for all convex functions $f$, $g$, and $h$, as well as specific parameter settings that provide a near-optimal convergence rate. When $g$ or $h$ is strongly convex, or when $f$ is smooth, we prove accelerated sublinear convergence rates. Finally, when $g$ or $h$ is strongly convex and $f$ is smooth, we prove a linear convergence rate. The proofs of convergence for algorithms incorporating momentum are often considered difficult and their intuition opaque \cite{lincoup}, but each of our proofs are simple extensions of Lemma \ref{lem:cvtele}, and adapting ACV to take advantage of additional smoothness and strong convexity only requires using different parameter settings in Algorithm \ref{alg:framework}.


\subsection{Convergence in the general setting}

Without any additional assumptions on the functions in \eqref{eq:main}, we show that Algorithm \ref{alg:acv} achieves a convergence rate of $\mathcal{O}(L/T^2 + \|A\|_{\textnormal{op}}/T)$. This convergence rate is optimal, and as $L/\|A\|_{\textnormal{op}}$ goes from $0$ to $\infty$, this rate interpolates between the rate of APGD and the Chambolle--Pock algorithm. As $k$ increases, the step sizes $\gamma_k$ and $\taustep_k$ also interpolate between the step sizes of accelerated proximal gradient descent up to a factor of $1/2$ (so $\taustep_k = \tfrac{k+1}{4 L}$) \cite{lincoup} 
and common choices of step sizes for the Chambolle--Pock algorithm up to a factor of $1/\sqrt{2}$ (so $\gamma_k = \taustep_k = \tfrac{1}{\sqrt{2} \|A\|_{\textnormal{op}}}$) \cite{chambollepock}.

\begin{center}
\begin{minipage}{0.975\linewidth}
\begin{algorithm}[H]
\caption{Accelerated Condat--V\~u: No strong convexity, $f$ is non-smooth}
\label{alg:acv}
  \begin{algorithmic}[1]
  \State Initialize $\primal_0, \momentumtwo_0 \in \mathcal{X}$ and $\dualit_0 \in \mathcal{Y}$. Set $\gamma_k$, $\taustep_k$, $\alpha_k$, and $\theta_k$ so that the constraints in \eqref{eq:genconst} are satisfied.
  \State 
  $\textnormal{ACV}(\primal_0,\dualit_0,\momentumtwo_0, \gamma_k,\taustep_k,\alpha_k,\theta_k,T)$.
  \State $w_{T} \leftarrow \alpha_{T} \dualit_{T} + (1 - \alpha_{T}) \momentumtwo_{T-1}$.
  \Ensure $(\momentumtwo_T, w_T)$.
  \end{algorithmic}
\end{algorithm}
\end{minipage}
\end{center}

This result significantly improves on the convergence rate of the unaccelerated Condat--V\~u algorithm, which is shown to be $\mathcal{O}((L + \|A\|_{\textnormal{op}})/T)$ \cite{CPrates}. It also matches the convergence rate of the accelerated algorithm of Chen \textit{et al}. \cite{chen13}, although the algorithm of Chen \textit{et al}. is only applicable when $g \equiv 0$.

\begin{theorem}
Suppose the parameters of Algorithm \ref{alg:acv} satisfy
\begin{equation}
\label{eq:genconst}
\begin{aligned}
    & \frac{\gamma_{k+1} (1 - \alpha_{k+1})}{\alpha_{k+1}} - \frac{\gamma_k}{\alpha_k} \le 0, ~  \frac{\gamma_{k+1}}{\taustep_{k+1}} - \frac{\gamma_k}{\taustep_k} \le 0, ~ 
    L \alpha_k \taustep_k + \gamma_k \taustep_k \|A\|_{\textnormal{op}}^2 \le 1, ~ \theta_k = \frac{\gamma_{k-1}}{\gamma_k}.
\end{aligned}
\end{equation}
%
After $T-1$ iterations, the following inequality holds:
\begin{equation}
\label{eq:cvx}
\begin{aligned}
    & PD_{\mathcal{B}_1 \times \mathcal{B}_2}(\momentumtwo_T, w_T) \le \frac{C \alpha_T}{\gamma_T},
\end{aligned}
\end{equation}
where
\begin{equation}
\label{eq:C}
    C \defeq \max_{(x',y') \in \mathcal{B}_1 \times \mathcal{B}_2} \frac{\gamma_0 (1 - \alpha_0)}{\alpha_0} ( \mathcal{L}(\momentumtwo_0, y') - \mathcal{L}(x', w_0) ) + \frac{\gamma_0}{2 \taustep_0} \|\primal_0 - x'\|^2 + \frac{1}{2} \|\dualit_0 - y'\|^2.
\end{equation}
\end{theorem}

\begin{proof}
We begin with the inequality of Lemma \ref{lem:cvtele} with $\mu_g = \mu_{f^*} = 0$.
\begin{equation}
\begin{aligned}
    & \frac{\gamma_k}{\alpha_k} ( \mathcal{L}(\momentumtwo_{k+1}, y) - \mathcal{L}(x, w_{k+1}) ) \\
    \le & \frac{\gamma_k (1 - \alpha_k)}{\alpha_k} ( \mathcal{L}(\momentumtwo_k, y) - \mathcal{L}(x, w_k) ) + \frac{\gamma_k}{2 \taustep_k} \left( L \alpha_k \taustep_k - 1 \right) \|\primal_{k+1} - \primal_k\|^2 \\
    & \quad + \frac{\theta_k^2 \gamma_k^2 \|A\|_{\textnormal{op}}^2}{2} \|\primal_k - \primal_{k-1}\|^2 + \frac{\gamma_k}{2 \taustep_k} (\|\primal_k - x\|^2 - \|\primal_{k+1} - x\|^2) + \frac{1}{2} \|\dualit_k - y\|^2 \\
    & \quad - \frac{1}{2} \|\dualit_{k+1} - y\|^2 + \gamma_k \theta_k \langle A (\primal_k - \primal_{k-1}), \dualit_k - y \rangle - \gamma_k \langle A (\primal_{k+1} - \primal_k), \dualit_{k+1} - y \rangle.
\end{aligned}
\end{equation}
We consider each term separately to ensure that this inequality telescopes over several iterations. The primal-dual gap terms are
\begin{equation}
\begin{aligned}
    & \frac{\gamma_k}{\alpha_k} ( \mathcal{L}(\momentumtwo_{k+1}, y) - \mathcal{L}(x, w_{k+1}) ) \le \frac{\gamma_k (1-\alpha_k)}{\alpha_k} ( \mathcal{L}(\momentumtwo_k, y) - \mathcal{L}(x, w_k) ) + \cdots.
\end{aligned}
\end{equation}
Taking the supremum in $(x,y)$ over the set $\mathcal{B}_1 \times \mathcal{B}_2$, these terms are non-negative, so they telescope as long as
\begin{equation}
\label{eq:cvx1}
    \frac{\gamma_{k+1}(1-\alpha_{k+1})}{\alpha_{k+1}} - \frac{\gamma_k}{\alpha_k} \le 0.
\end{equation}
For the terms involving $\|\primal_k - x\|^2$ to telescope, we must have 
\begin{equation}
\label{eq:cvx2}
    \frac{\gamma_{k+1}}{\tau_{k+1}} \le \frac{\gamma_k}{\tau_k},
\end{equation}
%
For the final inner-product term to telescope, we require 
\begin{equation}
\label{eq:cvx3}
    \theta_k = \frac{\gamma_{k-1}}{\gamma_k}.
\end{equation}
%
Taking the maximum of $(x,y) \in \mathcal{B}_1 \times \mathcal{B}_2$ and summing the resulting inequality from $k=0$ to $k = T-1$, we have
\begin{equation}
\begin{aligned}
    & \frac{\gamma_T}{\alpha_T} PD_{\mathcal{B}_1 \times \mathcal{B}_2}(\momentumtwo_T, w_T) \le C + \frac{\gamma_{T-1}}{2 \taustep_{T-1}} (L \alpha_{T-1} \taustep_{T-1} - 1) \|\primal_T - \primal_{T-1}\|^2 \\
    & + \sum_{k=0}^{T-2} \left( \frac{\gamma_k}{2 \taustep_k} \left( L \alpha_k \taustep_k + \gamma_k \taustep_k \|A\|_{\textnormal{op}}^2 - 1 \right) \|\primal_{k+1} - \primal_k\|^2 \right) \\
    & - \frac{\gamma_T}{2 \taustep_T} \|\primal_T - x'\|^2 - \frac{1}{2} \|\dualit_T - y'\|^2 - \gamma_{T-1} \langle A (\primal_T - \primal_{T-1}), \dualit_T - y' \rangle.
\end{aligned}
\end{equation}
We bound the final term using Young's inequality:
\begin{equation}
\begin{aligned}
    \gamma_{T-1} \langle A (\primal_T - \primal_{T-1}), \dualit_T - y' \rangle \le \frac{\gamma_{T-1}^2 \|A\|_{\textnormal{op}}^2}{2} \|\primal_T - \primal_{T-1}\|^2 + \frac{1}{2} \|\dualit_T - y'\|^2.
\end{aligned}
\end{equation}
The inequality then becomes
\begin{equation}
\begin{aligned}
    & \frac{\gamma_T}{\alpha_T} PD_{\mathcal{B}_1 \times \mathcal{B}_2}(\momentumtwo_T, w_T) \\
    \le & C + \Big(\sum_{k=0}^{T-1} \frac{\gamma_k}{2 \taustep_k} \left( L \alpha_k \taustep_k + \gamma_k \taustep_k \|A\|_{\textnormal{op}}^2 - 1 \right) \|\primal_{k+1} - \primal_k\|^2 \Big) - \frac{\gamma_T}{2 \taustep_T} \|\primal_T - x'\|^2.
\end{aligned}
\end{equation}
Using the final constraint on the parameters
\begin{equation}
\label{eq:cvx5}
    L \alpha_k \taustep_k + \gamma_k \taustep_k \|A\|_{\textnormal{op}}^2 \le 1,
\end{equation}
the terms involving $\|\primal_{k+1} - \primal_k\|^2$ are non-positive. Dropping all of the non-positive terms on the right, we have the result.
\end{proof}

\begin{remark}[Suggested Parameter Settings]
To achieve the desired convergence rates, we suggest the parameter settings
\begin{equation}\label{step_size_default}
\alpha_k = \frac{1}{k/2 + 1}, \quad \gamma_k = \tau_k = \frac{k+1}{\sqrt{2} \|A\|_{\textnormal{op}} k + 4 L}, \quad \textnormal{and} \quad \theta_k = \frac{\gamma_{k-1}}{\gamma_k}.
\end{equation}
With these values, inequality \eqref{eq:cvx} reads
\begin{equation}
    PD_{\mathcal{B}_1 \times \mathcal{B}_2}(\momentumtwo_T,w_T)
    \le \max_{(x',y') \in \mathcal{B}_1 \times \mathcal{B}_2} {  \frac{\sqrt{2} \|A\|_{\textnormal{op}} T + 4 L}{2 (1+\frac{T}{2})(1+T)}} \big(\|\primal_0 - x'\|^2 + \|\dualit_0 - y'\|^2 \big).
\end{equation}
This provides a convergence rate of $\mathcal{O}(L/T^2 + \|A\|_{\textnormal{op}} / T)$.

\begin{remark}
    The requirements on the parameters also show how ACV recovers the non-accelerated Condat--V\~u algorithm when $\alpha_k = 1$. In this case, the requirements on $\gamma_k$ and $\taustep_k$ reduce to $\gamma_k \taustep_k \|A\|_{\textnormal{op}}^2 \leq 1 - L \taustep_k$, which are the same conditions on the parameters derived in the original works of Condat and V\~u \cite{Condat2013,Vu2013condatvu}. In this case, we recover similar convergence guarantees on the ergodic sequence as the best-known for the Condat--V\~u algorithm.
\end{remark}
\end{remark}

\noindent If $g$ is strongly convex or if $f$ is smooth, these rates can be improved.


\subsection{Acceleration with strong convexity}

When $g$ is strongly convex, we prove a convergence rate of $\mathcal{O}(\tfrac{L}{\mu_g (1+\sqrt{\kappa_P})^{T}} + \tfrac{\|A\|_{\textnormal{op}}^2}{\mu_g T^2})$. This rate is optimal in the sense that it reduces to the convergence rate of accelerated proximal gradient descent when $\|A\|_{\textnormal{op}} = 0$, and it matches the rate of the Chambolle--Pock algorithm when there is no smooth component, {$h \equiv 0$ (implying $L = 0$)} \cite{chambollepock}. This result is much stronger than the convergence rate for Condat--V\~u in this setting, which has been shown to be $\mathcal{O}( (L + \|A\|_{\textnormal{op}}^2) / T^2)$. Our results also imply that it is possible to converge to a fixed-accuracy solution, one where the duality gap is $\mathcal{O}(\tfrac{\|A\|_{\textnormal{op}}^2}{\mu_g L})$, at a linear rate, because the linearly decaying term $\mathcal{O}(\tfrac{L}{\mu_g (1+\sqrt{\mu_g / L})^{T}})$ is dominant. For higher accuracy solutions, the algorithm transitions to a slower $\mathcal{O}(1/T^2)$ convergence rate.

\begin{center}
\begin{minipage}{0.975\linewidth}
\begin{algorithm}[H]
\caption{Accelerated Condat--V\~u for Strongly Convex Objectives}
\label{alg:acv2}
  \begin{algorithmic}[1]
  \State Initialize $\primal_0, v_0 \in \mathcal{X}$ and $\dualit_0 \in \mathcal{Y}$ and set $\gamma_0$, $\tau_0$, $\alpha_0$, and $\theta_0$ so that the constraints in \eqref{eq:acc_const} are satisfied. Set the number of warm-up iterations $T_0 \ge 1$.
  \State 
  $\textnormal{ACV}(\primal_0,\dualit_0,\momentumtwo_0,\gamma_0,\taustep_0,\alpha_0,\theta_0,T_0)$.
  \State Set $\gamma_k$, $\tau_k$, $\alpha_k$, and $\theta_k$ so that the constraints in \eqref{eq:acc_const_2} are satisfied.
  \State 
  $\textnormal{ACV}(\primal_{T_0},\dualit_{T_0},\momentumtwo_{T_0},\{\gamma_k\},\{\taustep_k\},\{\alpha_k\},\{\theta_k\},T)$
  \State $w_{T} \leftarrow \alpha_{T} \dualit_{T} + (1 - \alpha_{T}) \momentumtwo_{T-1}$.
  \Ensure $(\momentumtwo_T, w_T)$.
  \end{algorithmic}
\end{algorithm}
\end{minipage}
\end{center}

We call these two convergence phases the ``warm-up'' phase, where one can expect linear convergence, and the ``steady convergence'' phase, where the $\mathcal{O}(1/T^2)$ term in the convergence rate is dominant. These phases correspond to different parameter settings in Algorithm \ref{alg:acv2}. The parameters are held constant for the first $T_0 = \lfloor\sqrt{\frac{L}{\mu_g}} + \frac{\max\{\log(5 L/(2 \|A\|_{\textnormal{op}}^2) ),0\}}{\log(1+\sqrt{\mu_g / (4 L)})} \rfloor$ iterations to achieve linear convergence. After this, we use parameters similar to those in Algorithm \ref{alg:acv}.

\begin{theorem}
\label{thm:cvacc}
Suppose $g$ is $\mu_g$-strongly convex, let $T_0 \ge 1$, and let 
\begin{equation}
U_{T_0} \defeq \max \big\{ \max_{y' \in \mathcal{B}_2} \| \dualit_k - y'\|^2 \big\}_{k=0}^{T_0}.
\end{equation}
%
\emph{(Warm-Up Phase).} For the first $T_0$ iterations, set the parameters {  of Algorithm \ref{alg:acv2}} so that 
\begin{equation}
\label{eq:acc_const}
\frac{1}{\theta_0} \le \frac{1 - L \alpha_0 \taustep_0}{\gamma_0 \taustep_0 \|A\|_{\textnormal{op}}^2}, \quad \frac{1}{\theta_0} \le \frac{1}{1-\alpha_0}, \quad \frac{1}{\theta_0} \le 1 + \mu_g \taustep_0.
\end{equation}
%
After $T_0$ iterations,
\begin{equation}
\begin{aligned}
    & \frac{\gamma_0 (1-\alpha_0)}{\alpha_0} PD_{\mathcal{B}_1 \times \mathcal{B}_2}(v_{T_0}, w_{T_0}) + \frac{\gamma_0}{2 \taustep_0} \|\primal_{T_0} - \overline{x}\|^2 \le \theta^{T_0} C + \Big(\frac{1 - \theta^{T_0}}{1/\theta-1} \Big) \frac{\mu_g \tau_0}{2} U_{T_0} ,
\end{aligned}
\end{equation}
where $C$ is the constant defined in \eqref{eq:C}, and $\overline{x} \in \mathcal{B}_1$ is such that $\mathcal{L}(\momentumtwo_T, \overline{y}) - \mathcal{L}(\overline{x}, w_T) = PD_{\mathcal{B}_1 \times \mathcal{B}_2}(\momentumtwo_T,w_T)$ for some $\overline{y} \in \mathcal{B}_2$.

\vspace{5mm}

\noindent \emph{(Steady Convergence Phase).} Following the warm-up phase, set the parameters {  of Algorithm \ref{alg:acv2}} so that 
%
%
\begin{equation}
\label{eq:acc_const_2}
    \frac{\gamma_{k+1}(1-\alpha_{k+1})}{\alpha_{k+1}} - \frac{\gamma_k}{\alpha_k} \le 0, ~ \frac{\gamma_{k+1}}{\taustep_{k+1}} - \frac{\gamma_k (1 + \mu_g \taustep_k)}{\taustep_k} \le 0, ~ \frac{\|A\|_{\textnormal{op}}^2}{2} + \frac{L \alpha_k}{2 \gamma_k} - \frac{1}{2 \taustep_k \gamma_k} \le 0.
\end{equation}
After $T-1$ iterations of Algorithm \ref{alg:acv} with $T \ge T_0$, the following inequality holds:
\begin{equation}
\label{eq:strong}
\begin{aligned}
& PD_{\mathcal{B}_1 \times \mathcal{B}_2} (\momentumtwo_T, w_T) + \frac{\alpha_{T-1} (1 + \mu_g \taustep_{T-1})}{2\taustep_{T-1}} \|\primal_T - \overline{x}\|^2 \\
& \le \frac{\alpha_{T-1}}{\gamma_{T-1}} \Bigg(\theta^{T_0} C + \Big(\frac{1 - \theta^{T_0}}{\rho-1} \Big) \frac{\mu_g \tau_0}{2} U_{T_0} + \max_{y' \in \mathcal{B}_2} \frac{1}{2} \|\dualit_{T_0} - y'\|^2 \Bigg).
\end{aligned}
\end{equation}
%
%
\end{theorem}

\begin{proof}
We begin with the inequality of Lemma \ref{lem:cvtele}, setting $\mu_{f^*} = 0$.
\begin{equation}
\begin{aligned}
\label{eq:bothphase}
    & \frac{\gamma_k}{\alpha_k} ( \mathcal{L}(\momentumtwo_{k+1}, y) - \mathcal{L}(x, w_{k+1}) ) \\
    & \le \frac{\gamma_k (1 - \alpha_k)}{\alpha_k} ( \mathcal{L}(\momentumtwo_k, y) - \mathcal{L}(x, w_k) ) + \frac{\gamma_k}{2 \taustep_k} \left( L \alpha_k \taustep_k - 1 \right) \|\primal_{k+1} - \primal_k\|^2 + \frac{1}{2} \|\dualit_k - y\|^2 \\
    & + \frac{\theta_k^2 \gamma_k^2 \|A\|_{\textnormal{op}}^2}{2} \|\primal_k - \primal_{k-1}\|^2 + \frac{\gamma_k}{2 \taustep_k} (\|\primal_k - x\|^2 - (1 + \mu_g \taustep_k) \|\primal_{k+1} - x\|^2) \\
    & - \frac{1}{2} \|\dualit_{k+1} - y\|^2 + \gamma_k \theta_k \langle A (\primal_k - \primal_{k-1}), \dualit_k - y \rangle - \gamma_k \langle A (\primal_{k+1} - \primal_k), \dualit_{k+1} - y \rangle.
\end{aligned}
\end{equation}
%
We now consider the two phases of convergence separately, beginning with the warm-up phase where $k < T_0 
$.

\paragraph{Phase 1 (Warm-up)} 
Adding $\tfrac{\mu_g \tau_0}{2} \|y_{k+1} - y\|^2$ to both sides of \eqref{eq:bothphase},
%
%
\begin{equation}
\begin{aligned}
    & \frac{\gamma_0}{\alpha_0} ( \mathcal{L}(\momentumtwo_{k+1}, y) - \mathcal{L}(x, w_{k+1}) ) + \frac{\gamma_0}{2 \taustep_0} \left( 1 - L \alpha_0 \taustep_0 \right) \|\primal_{k+1} - \primal_k\|^2 \\
    & + \frac{\gamma_0 (1 + \mu_g \taustep_0)}{2 \taustep_0} \|\primal_{k+1} - x\|^2 + \frac{1 + \mu_g \taustep_0}{2} \|\dualit_{k+1} - y\|^2 + \gamma_0 \langle A (\primal_{k+1} - \primal_k), \dualit_{k+1} - y \rangle \\
    & \le \frac{\gamma_0 (1 - \alpha_0)}{\alpha_0} ( \mathcal{L}(\momentumtwo_k, y) - \mathcal{L}(x, w_k) ) + \frac{\theta_0^2 \gamma_0^2 \|A\|_{\textnormal{op}}^2}{2} \|\primal_k - \primal_{k-1}\|^2 + \frac{\gamma_0}{2 \taustep_0} \|\primal_k - x\|^2 \\
    & \quad + \frac{1}{2} \|\dualit_k - y\|^2 + \gamma_0 \theta_0 \langle A (\primal_k - \primal_{k-1}), \dualit_k - y \rangle + \frac{\mu_g \taustep_0}{2} \|\dualit_{k+1} - y\|^2.
\end{aligned}
\end{equation}
From this inequality, we would like to derive an inequality of the form $\rho E_{k+1} \le E_k$ for some functional $E_k$ and constant rate of decrease $\rho > 1$. The rate of decrease, $\rho$, is limited by the decrease of each term, providing the following conditions on $\rho$:
\begin{equation}
    \rho = 1 / \theta_0, \quad \rho \le \frac{1}{1 - \alpha_0}, \quad \rho \le 1 + \mu_g \tau_0, \quad \rho \le \frac{1 - L \alpha_0 \taustep_0}{\gamma_0 \taustep_0 \theta_0^2 \|A\|_{\textnormal{op}}^2}.
\end{equation}
%
%
%
{  The constraints on $\theta_0$ given by equation \eqref{eq:acc_const} ensure that these bounds are satisfied.} Hence, we have
\begin{equation}
\begin{aligned}
    & \rho \Bigg( \frac{\gamma_0(1-\alpha_0)}{\alpha_0} ( \mathcal{L}(\momentumtwo_{k+1}, y) - \mathcal{L}(x, w_{k+1}) ) + \frac{1}{2} \|\dualit_{k+1} - y\|^2 + \frac{\gamma_0}{2 \taustep_0} \|\primal_{k+1} - x\|^2 \\
    & \quad + \frac{\theta_0^2 \gamma_0^2 \|A\|_{\textnormal{op}}^2}{2} \|\primal_{k+1} - \primal_k\|^2 + \gamma_0 \theta_0 \langle A (\primal_{k+1} - \primal_k), \dualit_{k+1} - y \rangle \Bigg) \\
    \le & \frac{\gamma_0 (1 - \alpha_0)}{\alpha_0} ( \mathcal{L}(\momentumtwo_k, y) - \mathcal{L}(x, w_k) ) + \frac{1}{2} \|\dualit_k - y\|^2 + \frac{\gamma_0}{2 \taustep_0} \|\primal_k - x\|^2 \\
    & \quad + \frac{\theta_0^2 \gamma_0^2 \|A\|_{\textnormal{op}}^2}{2} \|\primal_k - \primal_{k-1}\|^2 + \gamma_0 \theta_0 \langle A (\primal_k - \primal_{k-1}), \dualit_k - y \rangle + \frac{\mu_g \taustep_0}{2} \|\dualit_{k+1} - y\|^2.
\end{aligned}
\end{equation}
Multiplying this inequality by $\rho^k$, summing from $k=0$ to $k=T_0-1$, {  and recalling that $x_{-1} = x_0$,} we obtain
\begin{equation}
\begin{aligned}
    & \rho^{T_0} \Bigg( \frac{\gamma_0 (1-\alpha_0)}{\alpha_0} ( \mathcal{L}(\momentumtwo_{T_0}, y) - \mathcal{L}(x, w_{T_0}) ) + \frac{1}{2} \|\dualit_{T_0} - y\|^2 + \frac{\gamma_0}{2 \taustep_0} \|\primal_{T_0} - x\|^2 \\
    & \quad + \frac{\theta_0^2 \gamma_0^2 \|A\|_{\textnormal{op}}^2}{2} \|\primal_{T_0} - \primal_{T_0-1}\|^2 + \gamma_0 \theta_0 \langle A (\primal_{T_0} - \primal_{T_0-1}), \dualit_{T_0} - y \rangle \Bigg) \\
    \le & \frac{\gamma_0 (1 - \alpha_0)}{\alpha_0} ( \mathcal{L}(\momentumtwo_0, y) - \mathcal{L}(x, w_0) ) + \frac{1}{2} \|\dualit_0 - y\|^2 + \frac{\gamma_0}{2 \taustep_0} \|\primal_0 - x\|^2 \\
    & \quad + \frac{\mu_g \tau_0}{2} \sum_{k=1}^{T_0} \rho^{k-1} \|\dualit_k - y\|^2.
\end{aligned}
\end{equation}
We bound the inner-product on the left using Young's inequality, obtaining
\begin{equation}
\begin{aligned}
    & \frac{\gamma_0 (1-\alpha_0)}{\alpha_0} ( \mathcal{L}(\momentumtwo_{T_0}, y) - \mathcal{L}(x, w_{T_0}) ) + \frac{\gamma_0}{2 \taustep_0} \|\primal_{T_0} - x\|^2 \\
    \le & \theta^{T_0} \Big( \frac{\gamma_0 (1 - \alpha_0)}{\alpha_0} ( \mathcal{L}(\momentumtwo_0, y) - \mathcal{L}(x, w_0) ) + \frac{1}{2} \|\dualit_0 - y\|^2 + \frac{\gamma_0}{2 \taustep_0} \|\primal_0 - x\|^2 \\
    & \quad + \frac{\mu_g \tau_0}{2} \sum_{k=1}^{T_0} \rho^{k-1} \|\dualit_k - y\|^2 \Big) .
\end{aligned}
\end{equation}
We can bound each term involving $\|\dualit_k - y\|^2$ by $U_{T_0}$. $U_{T_0}$ is finite because either $T_0$ is finite, or $\|A\|_{\textnormal{op}} = 0$ and the algorithm reduces to accelerated proximal gradient descent which converges. We have
\begin{equation}
\begin{aligned}
     \frac{\mu_g \tau_0}{2} \sum_{k=1}^{T_0} \rho^{k-1} \|\dualit_k - y\|^2 &\le \frac{\mu_g \tau_0}{2} \left(\sum_{k=1}^{T_0} \rho^{k-1} \right) U_{T_0} \\
    &\le \left(\frac{\rho^{T_0}-1}{\rho-1} \right) \frac{\mu_g \tau_0}{2} U_{T_0}.
\end{aligned}
\end{equation}
%
%
The final inequality is then
\begin{equation}
\label{eq:warmup}
\begin{aligned}
    & \frac{\gamma_0 (1-\alpha_0)}{\alpha_0} ( \mathcal{L}(\momentumtwo_{T_0}, y) - \mathcal{L}(x, w_{T_0}) ) + \frac{\gamma_0}{2 \taustep_0} \|\primal_{T_0} - x\|^2 \\
    \le & \theta^{T_0} \bigg( \frac{\gamma_0 (1 - \alpha_0)}{\alpha_0} ( \mathcal{L}(\momentumtwo_0, y) - \mathcal{L}(x, w_0) ) + \frac{1}{2} \|\dualit_0 - y\|^2 + \frac{\gamma_0}{2 \taustep_0} \|\primal_0 - x\|^2 \\
    & \quad + \Big(\frac{\rho^{T_0}-1}{\rho-1} \Big) \frac{\mu_g \tau_0}{2} U_{T_0} \bigg) .
\end{aligned}
\end{equation}
Taking the maximum of $(x,y) \in \mathcal{B}_1 \times \mathcal{B}_2$ produces the desired inequality. Because $\rho = 1/\theta$, the coefficient of the final term is approximately $\tfrac{\mu_g \tau_0}{2}$ and not decreasing in $T_0$, so asymptotic convergence during the warm-up phase is not guaranteed. The goal of the warm-up phase is to choose $T_0$ so that each of these terms is approximately equal, providing a good starting point for the steady convergence phase.

\vspace{5mm}

\begin{remark}[Suggested Parameter Settings]
    %
    %
    The parameter settings
    \begin{equation}
        \gamma_0 = \frac{\sqrt{\mu_g L}}{2 \|A\|_{\textnormal{op}}^2}, \quad \alpha_0 = \sqrt{\frac{\mu_g}{4 L}}, \quad \taustep_0 = \frac{1}{\sqrt{\mu_g L}}, \quad \theta_0 = \bigg(1+\sqrt{\frac{\mu_g}{4 L}} \bigg)^{-1},
    \end{equation}
    satisfy the warm-up phase constraints in \eqref{eq:acc_const}. With these parameter settings, Theorem \ref{thm:cvacc} states 
    %
    \begin{equation}
    \begin{aligned}
        & \frac{4 (1 - \sqrt{\frac{\mu_g}{4 L}})}{\mu_g} PD_{\mathcal{B}_1 \times \mathcal{B}_2}(\momentumtwo_{T_0}, w_{T_0}) +  \|\primal_{T_0} - \overline{x}\|^2 \\
        \le & \left(1 + \sqrt{\frac{\mu_g}{4 L}} \right)^{-T_0} \Big( \max_{(x',y') \in \mathcal{B}_1 \times \mathcal{B}_2} \frac{4 (1 - \sqrt{\frac{\mu_g}{4 L}})}{\mu_g} ( \mathcal{L}(\momentumtwo_0, y') - \mathcal{L}(x', w_0) ) \\
        &\quad + \frac{2 \|A\|_{\textnormal{op}}^2}{\mu_g L} \|\dualit_0 - y'\|^2 + \|\primal_0 - x'\|^2 \Big) + \frac{4 \|A\|_{\textnormal{op}}^2}{\mu_g L} U_{T_0}.
    \end{aligned}
    \end{equation}
    %
    %
    %
    If $\|A\|_{\textnormal{op}}$ is small, the final term involving $U_{T_0}$ is small and the first term dominates. The convergence rate appears linear until $\left(1 + \sqrt{\frac{\mu_g}{4 L}} \right)^{-T_0} \approx \|A\|_{\textnormal{op}}^2/L$. The number of warm-up iterations, $T_0$, should be chosen so that the terms on the top line of this inequality are of the same order as the terms on the bottom line. The choice $T_0 = \lfloor\sqrt{\frac{L}{\mu_g}} + \frac{\max\{\log(5 L/(2 \|A\|_{\textnormal{op}}^2) ),0\}}{\log(1+\sqrt{\mu_g / (4 L)})} \rfloor$ ensures that each of these terms is $\mathcal{O}(\frac{\|A\|_{\textnormal{op}}^2}{\mu_g L})$.
    In practice, the warm-up phase converges at a linear rate initially and plateaus close to the solution. For high-accuracy solutions, it is necessary to use both phases of the algorithm, but if a low-accuracy solution suffices, then running only the warm-up phase returns a suitable solution very quickly. Indeed, if a solution is desired with primal-dual gap $\epsilon$, and if $\frac{\|A\|_{\textnormal{op}}^2}{\mu_g L} < \epsilon$, then the bound above shows that we converge to a suitable solution at a linear rate.
\end{remark}

\paragraph{Phase 2 (Steady Convergence)} We now consider the iterations $k \ge T_0$ and interpret these iterations as restarting the algorithm from the points $\primal_{T_0}$ and $\dualit_{T_0}$ output from the warm-up phase. We show that each of the terms in inequality \eqref{eq:bothphase} telescopes over multiple iterations. The primal-dual gap terms are
\begin{equation}
\begin{aligned}
    & \frac{\gamma_k}{\alpha_k} ( \mathcal{L}(\momentumtwo_{k+1}, y) - \mathcal{L}(x, w_{k+1}) ) \le \frac{\gamma_k (1-\alpha_k)}{\alpha_k} ( \mathcal{L}(\momentumtwo_k, y) - \mathcal{L}(x, w_k) ) + \cdots.
\end{aligned}
\end{equation}
Taking the supremum in $(x,y)$ over the set $\mathcal{B}_1 \times \mathcal{B}_2$, these terms are non-negative, so they telescope as long as
\begin{equation}
\label{eq:steadyfirst}
    \frac{\gamma_{k+1}(1-\alpha_{k+1})}{\alpha_{k+1}} - \frac{\gamma_k}{\alpha_k} \le 0.
\end{equation}
Similarly, for the terms $\|x_k - x\|^2$ and $\langle A(x_k - x_{k-1}, y_k - y\rangle$ to telescope, we require
\begin{equation}
\label{eq:steadysecond}
    \frac{\gamma_{k+1}}{\taustep_{k+1}} - \frac{\gamma_k (1 + \mu_g \taustep_k)}{\taustep_k} \le 0 \quad \textnormal{and} \quad \theta = \frac{\gamma_{k-1}}{\gamma_k}, 
\end{equation}
respectively. With these conditions met, we can sum inequality \eqref{eq:telelast} from iteration $k=T_0$ to $k=T-1$:
\begin{equation}
\begin{aligned}
    & \frac{\gamma_{T-1}}{\alpha_{T-1}} PD_{\mathcal{B}_1 \times \mathcal{B}_2} (\momentumtwo_T, w_T) + \frac{\gamma_{T-1} (1 + \mu_g \taustep_{T-1})}{2\taustep_{T-1}} \|\primal_T - \overline{x}\|^2 \\
    & \le \max_{(x',y') \in \mathcal{B}_1 \times \mathcal{B}_2} \frac{\gamma_{T_0} (1 - \alpha_{T_0})}{\alpha_{T_0}} ( \mathcal{L}(\momentumtwo_{T_0}, y') - \mathcal{L}(x', w_{T_0}) ) \\
    & + \left(\sum_{k=T_0}^{T-1} \left(\frac{L \alpha_k \gamma_k}{2} - \frac{\gamma_k}{2 \taustep_k} \right) \|\primal_{k+1} - \primal_k\|^2 + \frac{\theta_k^2 \gamma_k^2 \|A\|_{\textnormal{op}}^2}{2} \|\primal_k - \primal_{k-1}\|^2 \right) \\
    & + \frac{\gamma_{T_0}}{2 \taustep_{T_0}} \|\primal_{T_0} - x'\|^2 + \frac{1}{2} \|\dualit_{T_0} - y'\|^2 - \frac{1}{2} \|\dualit_T - y'\|^2 \\
    & - \gamma_{T-1} \langle A (\primal_T - \primal_{T-1}), \dualit_T - y' \rangle.
\end{aligned}
\end{equation}
Because we have restarted the algorithm using the initial points $\primal_{T_0}$, we can use the convention $\primal_{T_0-1} = \primal_{T_0}$ to drop the first inner-product term just as we would say $\primal_{-1} = \primal_{0}$. We bound the final term using Young's inequality:
\begin{equation}
\begin{aligned}
    \gamma_{T-1} \langle A (\primal_T - \primal_{T-1}), \dualit_T - y' \rangle \le \frac{\gamma_{T-1}^2 \|A\|_{\textnormal{op}}^2}{2} \|\primal_T - \primal_{T-1}\|^2 + \frac{1}{2} \|\dualit_T - y'\|^2.
\end{aligned}
\end{equation}
This gives
\begin{equation}
\begin{aligned}
    & \frac{\gamma_{T-1}}{\alpha_{T-1}} PD_{\mathcal{B}_1 \times \mathcal{B}_2} (\momentumtwo_T, w_T) + \frac{\gamma_{T-1} (1 + \mu_g \taustep_{T-1})}{2\taustep_{T-1}} \|\primal_T - \overline{x}\|^2 \\
    & \le \max_{(x',y') \in \mathcal{B}_1 \times \mathcal{B}_2} \frac{\gamma_{T_0} (1 - \alpha_{T_0})}{\alpha_{T_0}} ( \mathcal{L}(\momentumtwo_{T_0}, y') - \mathcal{L}(x', w_{T_0}) ) \\
    & \quad + \sum_{k={  T_0}}^{T} \left(\frac{\gamma_{k-1}^2 \|A\|_{\textnormal{op}}^2}{2} + \frac{L \alpha_{k-1} \gamma_{k-1}}{2} - \frac{\gamma_{k-1}}{2 \taustep_{k-1}} \right) \|\primal_k - \primal_{k-1}\|^2 + \frac{\gamma_{T_0}}{2 \taustep_{T_0}} \|\primal_{T_0} - x'\|^2 \\
    & \quad + \frac{1}{2} \|\dualit_{T_0} - y'\|^2.
\end{aligned}
\end{equation}
In order to drop the terms involving $\|\primal_k - \primal_{k-1}\|^2$, we must have
\begin{equation}
\label{eq:steadythird}
    \frac{\gamma_{k-1}^2 \|A\|_{\textnormal{op}}^2}{2} + \frac{L \alpha_{k-1} \gamma_{k-1}}{2} - \frac{\gamma_{k-1}}{2 \taustep_{k-1}} \le 0, \quad \forall k \ge 1.
\end{equation}
With this condition satisfied, we drop the non-positive terms to obtain
%
\begin{equation}
\begin{aligned}
    & PD_{\mathcal{B}_1 \times \mathcal{B}_2} (\momentumtwo_T, w_T) + \frac{\alpha_{T-1} (1 + \mu_g \taustep_{T-1})}{2\taustep_{T-1}} \|\primal_T - \overline{x}\|^2 \\
    &\le \frac{\alpha_{T-1}}{\gamma_{T-1}} \max_{(x',y') \in \mathcal{B}_1 \times \mathcal{B}_2} \Bigg( \frac{\gamma_{T_0} (1 - \alpha_{T_0})}{\alpha_{T_0}} ( \mathcal{L}(\momentumtwo_{T_0}, y') - \mathcal{L}(x', w_{T_0}) ) + \frac{\gamma_{T_0}}{2 \taustep_{T_0}} \|\primal_{T_0} - x'\|^2 \\
    & \quad + \frac{1}{2} \|\dualit_{T_0} - y'\|^2 \Bigg) \\
\end{aligned}
\end{equation}
%
Using inequality \eqref{eq:warmup} to bound the terms on the top line
proves \eqref{eq:strong}.

\begin{remark}[Suggested Parameter Settings]
    The parameter settings
    \begin{equation}
        \gamma_k = \frac{\mu_g (k+4\sqrt{\mu_g/L})}{8 \|A\|_{\textnormal{op}}^2}, \quad \alpha_k = \frac{\mu_g}{4 \|A\|_{\textnormal{op}}^2 \gamma_k}, \quad \textnormal{and} \quad \taustep_k = \frac{1}{2 \|A\|_{\textnormal{op}}^2 \gamma_k}
    \end{equation}
    satisfy the constraints in \eqref{eq:acc_const_2} and approximately optimize the convergence rate. Inequality \eqref{eq:strong} then implies (by plugging in the suggested parameter choices for both phases):
    \begin{equation}
    \begin{aligned}
        & PD_{\mathcal{B}_1 \times \mathcal{B}_2} (\momentumtwo_T, w_T) + \frac{\mu_g}{4} \|\primal_T - \overline{x}\|^2 \\
        & \le \max_{(x',y') \in \mathcal{B}_1 \times \mathcal{B}_2} \frac{4 L \left(1 + \sqrt{\frac{\mu_g}{4 L}} \right)^{-T_0}}{(T + 4 \sqrt{L/\mu_g} - 1)^2} \Big( \frac{1 - \sqrt{\frac{\mu_g}{4 L}}}{\mu_g} ( \mathcal{L}(\momentumtwo_0, y') - \mathcal{L}(x', w_0) ) \\
        & + \frac{2 \|A\|_{\textnormal{op}}^2}{\mu_g L} \|\dualit_0 - y'\|^2 + \|\primal_0 - x'\|^2 \Big) + \frac{8 \|A\|_{\textnormal{op}}^2}{\mu_g (T + 4 \sqrt{L/\mu_g} - 1)^2)} (2 U_{T_0} + \|\dualit_{T_0} - y'\|^2).
    \end{aligned}
    \end{equation}
    With $T_0 \defeq \lfloor\sqrt{\frac{L}{\mu_g}} + \frac{\max\{\log(5 L/(2 \|A\|_{\textnormal{op}}^2) ),0\}}{\log(1+\sqrt{\mu_g / (4 L)})} \rfloor$, this convergence rate is $\mathcal{O}( \frac{L}{\mu_g (1 + \sqrt{\kappa_P})^T}$ $+ \frac{\|A\|_{\textnormal{op}}^2}{\mu_g T^2} )$.
\end{remark}

\end{proof}

\noindent Due to symmetry, it is easy to prove a similar result when $f^*$, instead of $g$, is $\mu_{f^*}$-strongly convex.

%

\begin{corollary}
\label{cor:cvacc2}
Suppose $\nabla f$ is $1/\mu_{f^*}$-Lipschitz continuous, let $T_0 \ge 1$, and let 
\begin{equation}
V_{T_0} \defeq \max \big\{ \max_{x' \in \mathcal{B}_2} \| x_k - x'\|^2 \big\}_{k=0}^{T_0}.
\end{equation}
%
\emph{(Warm-Up Phase).} For the first $T_0$ iterations, set the parameters {  of Algorithm \ref{alg:acv2}} so that 
\begin{equation}
\frac{1}{\theta_0} \le \frac{1 - L \alpha_0 \gamma_0}{\taustep_0 \gamma_0 \|A\|_{\textnormal{op}}^2}, \quad \frac{1}{\theta_0} \le \frac{1}{1-\alpha_0}, \quad \frac{1}{\theta_0} \le 1 + \mu_{f^*} \gamma_0.
\end{equation}
%
After $T_0$ iterations,
\begin{equation}
\begin{aligned}
    & \frac{\taustep_0 (1-\alpha_0)}{\alpha_0} PD_{\mathcal{B}_1 \times \mathcal{B}_2}(v_{T_0}, w_{T_0}) + \frac{\taustep_0}{2 \gamma_0} \|\primal_{T_0} - \overline{x}\|^2 \le \theta^{T_0} C' + \Big(\frac{1 - \theta^{T_0}}{1/\theta-1} \Big) \frac{\mu_g \taustep_0}{2} V_{T_0} ,
\end{aligned}
\end{equation}
where
\begin{equation}
C' \defeq \max_{(x',y') \in \mathcal{B}_1 \times \mathcal{B}_2} \frac{\taustep_0 (1 - \alpha_0)}{\alpha_0} ( \mathcal{L}(\momentumtwo_0, y') - \mathcal{L}(x', w_0) ) + \frac{1}{2} \|\primal_0 - x'\|^2 + \frac{\taustep_0}{2 \gamma_0} \|\dualit_0 - y'\|^2,
\end{equation}
{  and $\overline{x} \in \mathcal{B}_1$ and $\overline{y} \in \mathcal{B}_2$ are such that} $\mathcal{L}(\momentumtwo_T, \overline{y}) - \mathcal{L}(\overline{x}, w_T) = PD_{\mathcal{B}_1 \times \mathcal{B}_2}(\momentumtwo_T,$ $w_T)$.

\vspace{5mm}

\noindent \emph{(Steady Convergence Phase).} Following the warm-up phase, set the parameters {  of Algorithm \ref{alg:acv2}} so that 
%
%
\begin{equation}
    \frac{\taustep_{k+1}(1-\alpha_{k+1})}{\alpha_{k+1}} - \frac{\taustep_k}{\alpha_k} \le 0, ~ \frac{\taustep_{k+1}}{\gamma_{k+1}} - \frac{\taustep_k (1 + \mu_{f^*} \gamma_k)}{\gamma_k} \le 0, ~ {  \frac{\|A\|_{\textnormal{op}}^2}{2} + \frac{L \alpha_k}{2 \taustep_k} - \frac{1}{2 \gamma_k \taustep_k} } \le 0.
\end{equation}
After $T-1$ iterations of Algorithm \ref{alg:acv} with $T \ge T_0$, the following inequality holds:
\begin{equation}
\begin{aligned}
& PD_{\mathcal{B}_1 \times \mathcal{B}_2} (\momentumtwo_T, w_T) + \frac{\alpha_{T-1} (1 + \mu_{f^*} \gamma_{T-1})}{2\gamma_{T-1}} \|\dualit_T - \overline{y}\|^2 \\
& \le \frac{\alpha_{T-1}}{\taustep_{T-1}} \Bigg( \theta^{T_0} {  C'} + \Big(\frac{1 - \theta^{T_0}}{\rho-1} \Big) \frac{\mu_{f^*} \tau_0}{2} V_{T_0} + \frac{1}{2} \|\primal_{T_0} - x'\|^2 \Bigg).
\end{aligned}
\end{equation}

\end{corollary}

\begin{proof}
We set $\mu_g = 0$ in equation \eqref{eq:telelast} and  multiply by $\taustep_k / \gamma_k$.
\begin{equation}
\begin{aligned}
    & \frac{\taustep_k}{\alpha_k} ( \mathcal{L}(\momentumtwo_{k+1}, y) - \mathcal{L}(x, w_{k+1}) ) \\
    \le & \frac{\taustep_k (1 - \alpha_k)}{\alpha_k} ( \mathcal{L}(\momentumtwo_k, y) - \mathcal{L}(x, w_k) ) + \frac{L \alpha_k \taustep_k - 1}{2} \|\primal_{k+1} - \primal_k\|^2 \\
    & \quad + \frac{\theta_k^2 \taustep_k \gamma_k \|A\|_{\textnormal{op}}^2}{2} \|\primal_k - \primal_{k-1}\|^2 + \frac{1}{2} (\|\primal_k - x\|^2 - \|\primal_{k+1} - x\|^2) + \frac{\taustep_k}{2 \gamma_k} \|\dualit_k - y\|^2 \\
    & \quad - \frac{\taustep_k (1 + \mu_{f^*} \gamma_k)}{2 \gamma_k} \|\dualit_{k+1} - y\|^2 + \taustep_k \theta_k \langle A (\primal_k - \primal_{k-1}), \dualit_k - y \rangle \\
    & \quad - \taustep_k \langle A (\primal_{k+1} - \primal_k), \dualit_{k+1} - y \rangle.
\end{aligned}
\end{equation}
This is similar to inequality \eqref{eq:bothphase} in the proof of Theorem \ref{thm:cvacc}, but $\mu_g$ is replaced with $\mu_{f^*}$, and $\taustep_k$ is interchanged with $\gamma_k$. Adjusting the parameters accordingly, we follow the proof of Theorem \ref{thm:cvacc} to obtain analogous convergence results in both phases of the algorithm.
\end{proof}

\subsection{Acceleration for smooth and strongly convex objectives}

If both $g$ and $f^*$ are strongly convex, then we prove a linear convergence rate that improves on the linear rate of Condat--V\~u. These rates depend on two condition numbers.
When $f^*$ is $\mu_{f^*}$-strongly convex, then $f$ is differentiable and its gradient is $1/\mu_{f^*}$-Lipschitz continuous. It follows that the gradient of $f \circ A$ is $\|A\|_{\textnormal{op}}^2/\mu_{f^*}$-Lipschitz continuous. The quantities $\kappa_P = L / \mu_g$ and $\kappa_{PD} = \|A\|_{\textnormal{op}}^2/(\mu_{f^*}\mu_g)$ are then the condition numbers of the objectives $h + g$ and $f \circ A + g$, respectively. Because $f$ is smooth, it is possible to apply APGD to the objective $f \circ A + g + h$ using only the proximal operator of $g$, achieving a convergence rate of $\mathcal{O}( (1 + \sqrt{1/\kappa_P} + \sqrt{1/\kappa_{PD}})^{-T} )$.\footnote{This is because the convergence rate of APGD is $\mathcal{O}((1 + \sqrt{\mu_g/L'})^{-T})$, where $L' = L + \|A\|^2_{\textnormal{op}}/\mu_{f^*}$ 
is the Lipschitz constant of the gradient of $f \circ A + h$} Hence, there is no justification to use the non-accelerated Condat--V\~u algorithm, which has a convergence rate of $\mathcal{O}( (1 + \sqrt{1/\kappa_{PD}} + 1/\kappa_P)^{-T})$ \cite{CPrates}. 

%
%
In contrast, ACV is a competitive choice for minimizing these objectives. In this section, we show that ACV achieves a convergence rate of $\mathcal{O}( (1 + \sqrt{1/\kappa_{PD}} + \sqrt{1/\kappa_P})^{-T} )$. 


\begin{center}
\begin{minipage}{0.975\linewidth}
\begin{algorithm}[H]
\caption{Accelerated Condat--V\~u for Stongly Convex and Smooth Objectives}
\label{alg:acv4}
  \begin{algorithmic}[1]
  \State Initialize $\primal_0, \momentumtwo_0 \in \mathcal{X}$, $\dualit_0 \in \mathcal{Y}$ and set $\gamma, \tau, \alpha$, and $\theta$ so that the constraints in \eqref{eq:strongconstthm} are satisfied.
  \State 
  $\textnormal{ACV}(\primal_0,\dualit_0,v_0,\gamma,\taustep,\alpha,\theta,T)$
  \State $w_{T} \leftarrow \alpha_{T} \dualit_{T} + (1 - \alpha_{T}) \momentumtwo_{T-1}$.
  \Ensure $(\momentumtwo_T, w_T)$.
  \end{algorithmic}
\end{algorithm}
\end{minipage}
\end{center}


\begin{theorem}
\label{thm:cvlin}
Suppose $g$ is $\mu_g$-strongly convex and $f^*$ is $\mu_{f^*}$-strongly convex.
After $T-1$ iterations of Algorithm \ref{alg:acv4} with parameters satisfying the constraints 
\begin{equation}
\label{eq:strongconstthm}
    \frac{1}{\theta} \le \frac{1}{1 - \alpha}, \quad \frac{1}{\theta} \le 1 + \mu_{f^*} \gamma, \quad \frac{1}{\theta} \le 1 + \mu_g \tau, \quad \frac{1}{\theta} \le \frac{1 - L \alpha \taustep}{\gamma \taustep \theta^2 \|A\|_{\textnormal{op}}^2},
\end{equation}
%
%
the following inequality holds:
\begin{equation}
\label{eq:strongres}
\begin{aligned}
    & \frac{\gamma (1-\alpha)}{\alpha} PD_{\mathcal{B}_1 \times \mathcal{B}_2}(\momentumtwo_T,w_T) + \frac{\gamma}{2 \taustep} \|\primal_T - \overline{x}\|^2 \le \theta^T C,
\end{aligned}
\end{equation}
where $C$ is as defined in \eqref{eq:C}, and $\overline{x} \in \mathcal{B}_1$ is such that $\mathcal{L}(\momentumtwo_T, \overline{y}) - \mathcal{L}(\overline{x}, w_T) = PD_{\mathcal{B}_1 \times \mathcal{B}_2}(\momentumtwo_T,$ $w_T)$ for some $\overline{y} \in \mathcal{B}_2$.
\end{theorem}

\begin{proof}
We begin with the inequality of Lemma \ref{lem:cvtele} with $\mu_g, \mu_{f^*} > 0$. Because the parameters are fixed, we drop their dependence on $k$.
\begin{equation}
\begin{aligned}
    & \frac{\gamma}{\alpha} ( \mathcal{L}(\momentumtwo_{k+1}, y) - \mathcal{L}(x, w_{k+1}) ) \\
    & \le \frac{\gamma (1 - \alpha)}{\alpha} ( \mathcal{L}(\momentumtwo_k, y) - \mathcal{L}(x, w_k) ) + \frac{\gamma}{2 \taustep} \left( L \alpha \taustep - 1 \right) \|\primal_{k+1} - \primal_k\|^2 + \frac{1}{2} \|\dualit_k - y\|^2 \\
    & + \frac{\theta^2 \gamma^2 \|A\|_{\textnormal{op}}^2}{2} \|\primal_k - \primal_{k-1}\|^2 + \frac{\gamma}{2 \taustep} (\|\primal_k - x\|^2 - (1 + \mu_g \taustep) \|\primal_{k+1} - x\|^2) \\
    & - \frac{1 + \mu_{f^*} \gamma}{2} \|\dualit_{k+1} - y\|^2 + \gamma \theta \langle A (\primal_k - \primal_{k-1}), \dualit_k - y \rangle - \gamma \langle A (\primal_{k+1} - \primal_k), \dualit_{k+1} - y \rangle.
\end{aligned}
\end{equation}

Rearranging,
\begin{equation}
\begin{aligned}
    & \frac{\gamma}{\alpha} ( \mathcal{L}(\momentumtwo_{k+1}, y) - \mathcal{L}(x, w_{k+1}) ) + \frac{1 + \mu_{f^*} \gamma}{2} \|\dualit_{k+1} - y\|^2 + \frac{\gamma (1 + \mu_g \taustep)}{2 \taustep} \|\primal_{k+1} - x\|^2 \\
    & \quad + \frac{\gamma}{2 \taustep} \left( 1 - L \alpha \taustep \right) \|\primal_{k+1} - \primal_k\|^2 + \gamma \langle A (\primal_{k+1} - \primal_k), \dualit_{k+1} - y \rangle \\
    \le & \frac{\gamma (1 - \alpha)}{\alpha} ( \mathcal{L}(\momentumtwo_k, y) - \mathcal{L}(x, w_k) ) + \frac{1}{2} \|\dualit_k - y\|^2 + \frac{\gamma}{2 \taustep} \|\primal_k - x\|^2 \\
    & \quad + \frac{\theta^2 \gamma^2 \|A\|_{\textnormal{op}}^2}{2} \|\primal_k - \primal_{k-1}\|^2 + \gamma \theta \langle A (\primal_k - \primal_{k-1}), \dualit_k - y \rangle.
\end{aligned}
\end{equation}
Just as in the analysis of the warm-up phase in Algorithm \ref{alg:acv2}, we would like to derive an inequality of the form $\rho E_{k+1} \le E_k$ for some functional $E_k$ and constant rate of decrease $\rho > 1$. The rate of decrease, $\rho$, is limited by the decrease of each term, giving the constraints
\begin{equation}
\label{eq:strongconst}
    \rho \le \frac{1}{1 - \alpha}, \quad \rho \le 1 + \mu_{f^*} \gamma, \quad \rho \le 1 + \mu_g \tau, \quad \rho \le \frac{1 - L \alpha \taustep}{\gamma \taustep \theta^2 \|A\|_{\textnormal{op}}^2}, \quad \rho = 1 / \theta.
\end{equation}
%
%
With these constraints satisfied, we have the inequality
\begin{equation}
\begin{aligned}
    & \rho \Bigg( \frac{\gamma(1-\alpha)}{\alpha} ( \mathcal{L}(\momentumtwo_{k+1}, y) - \mathcal{L}(x, w_{k+1}) ) + \frac{1}{2} \|\dualit_{k+1} - y\|^2 + \frac{\gamma}{2 \taustep} \|\primal_{k+1} - x\|^2 \\
    & \quad + \frac{\theta^2 \gamma^2 \|A\|_{\textnormal{op}}^2}{2} \|\primal_{k+1} - \primal_k\|^2 + \gamma \theta \langle A (\primal_{k+1} - \primal_k), \dualit_{k+1} - y \rangle \Bigg) \\
    \le & \frac{\gamma (1 - \alpha)}{\alpha} ( \mathcal{L}(\momentumtwo_k, y) - \mathcal{L}(x, w_k) ) + \frac{1}{2} \|\dualit_k - y\|^2 + \frac{\gamma}{2 \taustep} \|\primal_k - x\|^2 \\
    & \quad + \frac{\theta^2 \gamma^2 \|A\|_{\textnormal{op}}^2}{2} \|\primal_k - \primal_{k-1}\|^2 + \gamma \theta \langle A (\primal_k - \primal_{k-1}), \dualit_k - y \rangle.
\end{aligned}
\end{equation}
Multiplying this inequality by $\rho^k$ and summing from $k=0$ to $k=T-1$, we obtain
\begin{equation}
\begin{aligned}
\label{eq:adj}
    & \rho^T \Bigg( \frac{\gamma (1-\alpha)}{\alpha} ( \mathcal{L}(\momentumtwo_T, y) - \mathcal{L}(x, w_T) ) + \frac{1}{2} \|\dualit_T - y\|^2 + \frac{\gamma}{2 \taustep} \|\primal_T - x\|^2 \\
    & \quad + \frac{\theta^2 \gamma^2 \|A\|_{\textnormal{op}}^2}{2} \|\primal_T - \primal_{T-1}\|^2 + \gamma \theta \langle A (\primal_T - \primal_{T-1}), \dualit_T - y \rangle \Bigg) \\
    \le & \frac{\gamma (1 - \alpha)}{\alpha} ( \mathcal{L}(\momentumtwo_0, y) - \mathcal{L}(x, w_0) ) + \frac{1}{2} \|\dualit_0 - y\|^2 + \frac{\gamma}{2 \taustep} \|\primal_0 - x\|^2.
\end{aligned}
\end{equation}
We bound the inner-product on the left using Young's inequality, obtaining
\begin{equation}
\begin{aligned}
    & \frac{\gamma (1-\alpha)}{\alpha} ( \mathcal{L}(\momentumtwo_T, y) - \mathcal{L}(x, w_T) ) + \frac{\gamma}{2 \taustep} \|\primal_T - x\|^2 \\
    \le &  \theta^T \Big( \frac{\gamma}{2 \taustep} \|\primal_0 - x\|^2 + \frac{1}{2} \|\dualit_0 - y\|^2 + \frac{\gamma (1 - \alpha)}{\alpha} ( \mathcal{L}(\momentumtwo_0, y) - \mathcal{L}(x, w_0) ) \Big).
\end{aligned}
\end{equation}
Taking the supremum of $(x,y) \in \mathcal{B}_1 \times \mathcal{B}_2$ proves the assertion.
%
\end{proof}
\begin{remark}[Suggested Parameter Settings]
    Let $\overline{L} = \|A\|_{\textnormal{op}}^2/\mu_{f^*} +  L$ be the Lipschitz constant of $f \circ A + h$. The choices
    \begin{equation}\label{scx_recom_step_sizes}
    \begin{aligned}
        \gamma = \sqrt{\frac{\mu_g}{\mu_{f^*}^2 \overline{L}}}, \quad \taustep = \sqrt{\frac{1}{\overline{L} \mu_g}}, \quad \textnormal{and} \quad \alpha = \sqrt{\frac{\mu_g}{\overline{L}}}.
    \end{aligned}
    \end{equation}
    approximately maximize the convergence rate,
    \begin{equation}
        \rho = 1 + \sqrt{\frac{\mu_g}{\overline{L}}}.
    \end{equation}
    These parameter settings satisfy all of the constraints in \eqref{eq:strongconst}. The only non-trivial constraint to check is
    { 
    \begin{align*}
        \frac{1 - L \alpha \taustep}{\gamma \taustep \theta^2 \|A\|_{\textnormal{op}}^2} &= \frac{1 - L/\overline{L}}{ (1/(\mu_{f^*} \overline{L})) \theta^2 \|A\|_{\textnormal{op}}^2} \\
        &= (1 - L / \overline{L}) (\mu_{f^*} \overline{L}) \rho^2 / \|A\|_{\textnormal{op}}^2 \\
        &= \left(1 - \frac{L}{\|A\|_{\textnormal{op}}^2 / \mu_{f^*} + L}\right) \left(\frac{\|A\|_{\textnormal{op}}^2 + L \mu_{f^*}}{\|A\|_{\textnormal{op}}^2}\right) \rho^2 \\
        &= \left(1 - \frac{1}{\|A\|_{\textnormal{op}}^2 / (L \mu_{f^*}) + 1}\right) \left(1 + \frac{L \mu_{f^*}}{\|A\|_{\textnormal{op}}^2}\right) \rho^2 \\
        &= \rho^2 \\
        &\ge \rho.
    \end{align*}
    }
    With these parameter settings, inequality \eqref{eq:strongres} reads
    \begin{equation}
    \begin{aligned}
        & \frac{\gamma (1-\alpha)}{\alpha} PD_{\mathcal{B}_1 \times \mathcal{B}_2}(\momentumtwo_T,w_T) + \frac{\gamma}{2 \tau} \|\primal_T - \overline{x}\|^2 \le {\left(1 + \min \{ \sqrt{1/(2 \kappa_{PD})}, \ \sqrt{1/(2 \kappa_P)} \} \right)^{-T}} C,
    \end{aligned}
    \end{equation}
    where $C$ is defined in \eqref{eq:C}. This proves the desired convergence rate of $\mathcal{O}( (1 + \min \sqrt{1/\kappa_{PD}} + \sqrt{1/\kappa_P})^{-T})$.
\end{remark}
\begin{remark}[Convergence of the Dual Variable]
    As stated, Theorem \ref{thm:cvlin} provides a linear convergence rate on the primal variables, but not on the dual variables. Given that the primal iterates converge, a minor adjustment to \eqref{eq:adj} provides a rate on the convergence of the dual iterates as well. Applying Young's inequality to \eqref{eq:adj}, for any $\epsilon > 0$, we have
    \begin{equation}
    \begin{aligned}
        & \frac{\gamma (1-\alpha)}{\alpha} ( \mathcal{L}(\momentumtwo_T, y) - \mathcal{L}(x, w_T) ) + \frac{\gamma}{2 \taustep} \|\primal_T - x\|^2 + \frac{1 - \epsilon}{2} \|\dualit_T - y\|^2 \\
        & \quad - \frac{\theta^2 \gamma^2 \|A\|_{\textnormal{op}}^2 (1 - \epsilon)}{2} \|x_T - x_{T-1}\|^2 \\
        \le &  \theta^T \Big( \frac{\gamma}{2 \taustep} \|\primal_0 - x\|^2 + \frac{1}{2} \|\dualit_0 - y\|^2 + \frac{\gamma (1 - \alpha)}{\alpha} ( \mathcal{L}(\momentumtwo_0, y) - \mathcal{L}(x, w_0) ) \Big).
    \end{aligned}
    \end{equation}
    We have traded a term of the form $\|\primal_T - \primal_{T-1}\|^2$ for a term of the form $\|\dualit_T - y\|^2$. Because Theorem \ref{thm:cvlin} guarantees convergence of the primal iterates, $\primal_{T-1}$ will eventually lie in $\mathcal{B}_1$, so after maximizing $(x,y)$ over the set $\mathcal{B}_1 \times \mathcal{B}_2$, the term involving $\|\primal_T - \primal_{T-1}\|^2$ can be made small compared to $\| \primal_T - x \|^2$. This guarantees a linear convergence rate on the primal and dual iterates.
\end{remark}
We now demonstrate the superiority of the Accelerated Condat--V\~u algorithm in numerical experiments.

\begin{figure}[!t]
\centering
\subfloat[\texttt{australian}]{ \includegraphics[width=0.45\linewidth]{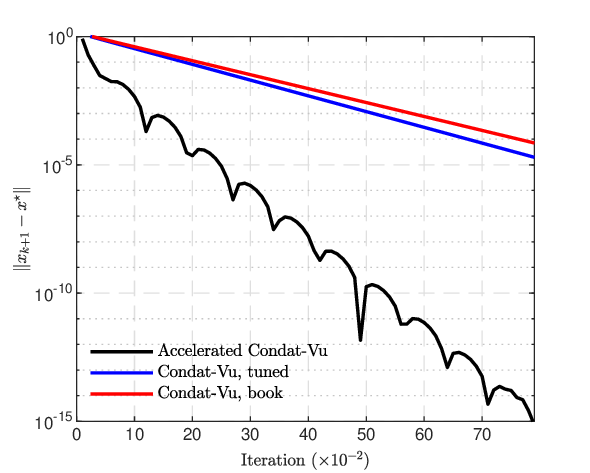}}
\subfloat[\texttt{mushrooms}]{ \includegraphics[width=0.45\linewidth]{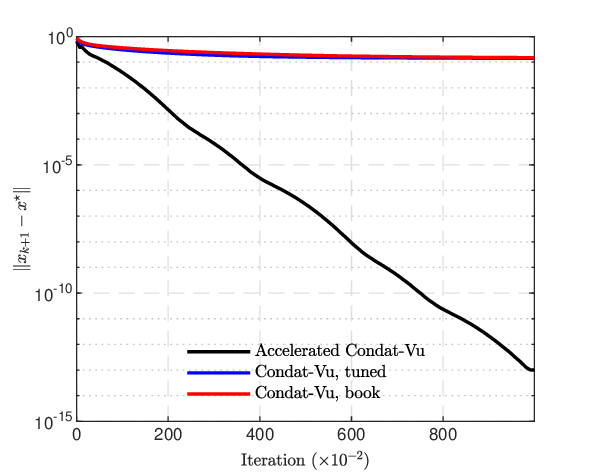} } 

\subfloat[\texttt{phishing}]{ \includegraphics[width=0.45\linewidth]{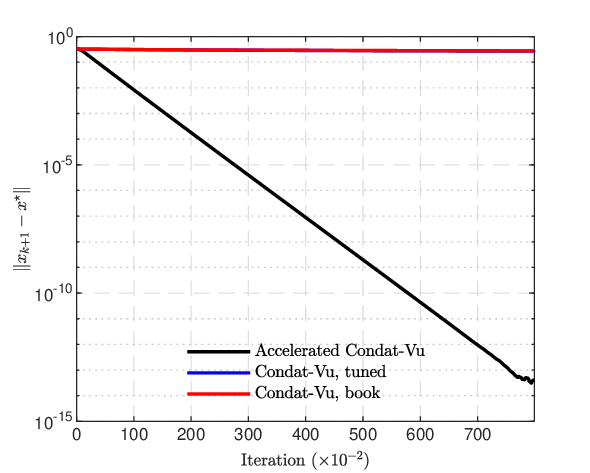} }   
\hspace{1pt}
\subfloat[\texttt{ijcnn1}]{ \includegraphics[width=0.45\linewidth]{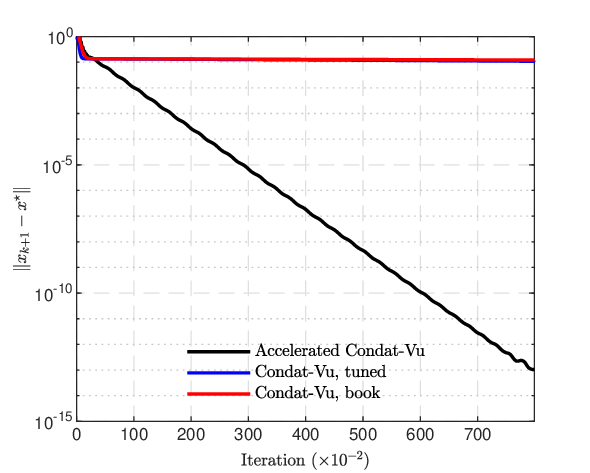} }  \\
\caption{Performance comparison of the proposed Accelerated Condat--V\~u, Condat--V\~u with standard parameters settings, and Condat--V\~u with tuned parameters on problem \eqref{eq:ggfen}. These results show that in the regime where $L \gg \|A\|_{\textnormal{op}}^2 / \mu_{f^*}$, ACV significantly outperforms CV, matching our theoretical convergence rates.  }
\label{fig:cv-elastic-net}
\end{figure}

\begin{figure}[!t]
\centering
\subfloat[\texttt{australian}]{ \includegraphics[width=0.45\linewidth]{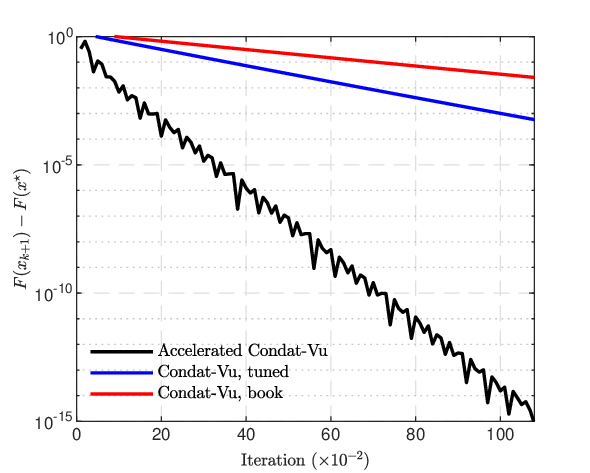}}
\subfloat[\texttt{mushrooms}]{ \includegraphics[width=0.45\linewidth]{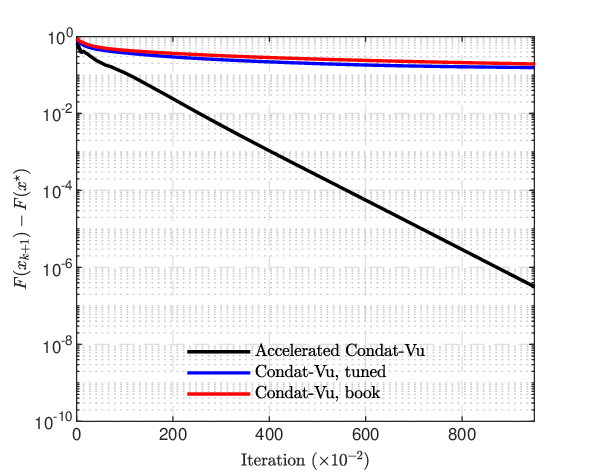} } 

\subfloat[\texttt{phishing}]{ \includegraphics[width=0.45\linewidth]{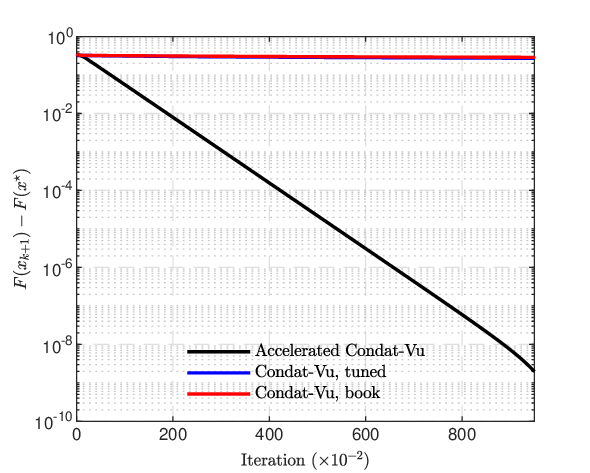} }   
\hspace{1pt}
\subfloat[\texttt{ijcnn1}]{ \includegraphics[width=0.45\linewidth]{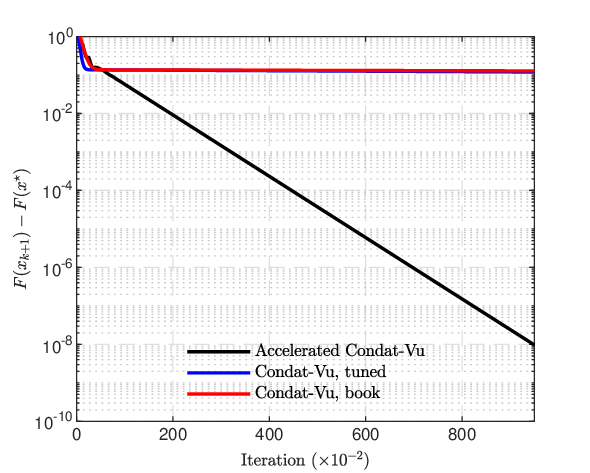} }  \\
\caption{
Performance comparison of ACV, Condat--V\~u with standard parameters settings \cite[Eq.\! 48]{CPrates}, and Condat--V\~u with tuned parameters on problem \eqref{eq:ggfen} with $\lambda_3 \to \infty$. These results show that in the regime where $L \gg  \|A\|_{\textnormal{op}}^2$, ACV significantly outperforms CV and matches our theoretical convergence rates.}
\label{fig:cv-elastic-net-nonsmooth}
\end{figure}

\section{Numerical Experiments}
\label{sec:ex}

We test the ACV algorithm using a common problem in statistical modeling and a common problem in image processing. Our experiments show that the Accelerated Condat--V\~u algorithm offers significant performance benefits when $L/\|A\|_{\textnormal{op}}$ is large, matching the near-optimal convergence rates predicted by our theory. When this ratio is small, ACV performs similarly to the non-accelerated Condat--V\~u algorithm that already achieves near-optimal convergence rates.

\subsection{Fused elastic net with smoothing}

We formulate the fused elastic net problem as
\begin{equation}
\label{eq:ggfen}
    \min_{x \in \mathbb{R}^d} \quad \frac{1}{2} \|W x - b\|^2 + \lambda_1 \beta \|x\|_1 + \frac{\lambda_1 (1-\beta)}{2} \|x\|^2 + \lambda_2 J(F x).
\end{equation}
Here, $b$ is a vector of labels; $W$ is a data matrix; $\lambda_1, \lambda_2, \beta \ge 0$ are tuning parameters; and $F$ is a matrix encoding relationships among the features. The regularizer $J$ is the \emph{Moreau-Yosida envelope} of the $\ell_1$-norm (also known as the \emph{Huber} norm \cite{huber}):
\begin{equation}
\label{eq:J}
    J(F x) \defeq \big(\|\cdot \|_1 \square \frac{\lambda_3}{2} \|\cdot\|^2 \big) (F x) = \min_{u \in \mathbb{R}^d} \|F x - u\|_1 + \frac{\lambda_3}{2} \|u\|^2,
\end{equation}
where $\square$ denotes the infimal convolution and $\lambda_3 \ge 0$ is a tuning parameter.

Problem \eqref{eq:ggfen} is a linear regression problem with regularization to promote sparsity in $x$ and preserve correlations among features. The matrix $F$ is defined so that
\begin{equation}
    \|F x\|_1 = \sum_{(i,j) \in \mathcal{I}} |x_i - x_j|,
\end{equation}
where $\mathcal{I} \subset \{1,\cdots,d\} \times \{1,\cdots,d\}$ is an index set containing the top 10\% most correlated pairs of features in $W$ (we defer to \cite{sadmm} for a more general description of $F$ for the graph-guided LASSO problem). Problem \eqref{eq:ggfen} contains many popular problems as special cases. With $\beta = 1$ and $\lambda_2 = 0$, this problem reduces to the LASSO problem \cite{lasso}, and we recover elastic net regularization when $\beta$ varies \cite{elasticnet}. With $\beta = 1$ and $\lambda_3 \to \infty$, this is the graph-guided fused LASSO problem \cite{gglasso,fusedlasso}, and with $\lambda_1 = 0$, our regularizers reduce to a single Huber-type penalty \cite{huber}. Letting $\lambda_3$ vary smooths the $\|F x\|_1$ penalty for easier optimization; we test the smoothed problem in this section and the non-smoothed problem in the following section.

In the Condat--V\~u algorithms, we let 
\begin{equation}
    f(y) = \lambda_2 J(y), \quad g(x) = \lambda_1 \beta \|x\|_1 + \frac{\lambda_1 (1-\beta)}{2} \|x\|^2, \quad h(x) = \tfrac{1}{2} \|W x - b\|^2.
\end{equation}
The conjugate $f^*(y) = \iota_{\|\cdot\|_{\infty} \le \lambda_2}(y) + \tfrac{1}{2 \lambda_2 \lambda_3} \|y\|^2$ has the following proximal operator, defined element-wise:
\begin{equation}
    (\prox_{\eta f^*}(z))_i = \tfrac{z_i}{1 + \eta / (\lambda_2 \lambda_3)} \cdot \min\{\tfrac{\eta (1 + \eta / (\lambda_2 \lambda_3))}{|z_i| \lambda_2}, 1\}.
\end{equation}
The proximal operator of $g$ is $(\prox_{\eta g}(z))_i = \textnormal{Soft}(\tfrac{z_i}{1 + \eta \lambda_1 (1-\beta)}, \tfrac{\eta \lambda_1 \beta}{1+\eta \lambda_1 (1-\beta)})$, where $\textnormal{Soft}(z,\lambda)$ $\defeq \textnormal{sgn}(z_i) \max \{ |z_i| - \lambda, 0 \}$ is the soft-thresholding operator (and $\textnormal{sgn}(\cdot)$ is the sign function).



Because $g$ is strongly convex with constant $\mu_g = \lambda_1 (1-\beta)$, and $f^*$ is strongly convex with constant $\mu_{f^*} = \frac{1}{\lambda_2 \lambda_3}$, we can use the linearly convergent ACV algorithm in Algorithm \ref{alg:acv4}. 
%
We run this experiment using four standard benchmark datasets from \texttt{LibSVM}\footnote{\url{https://www.csie.ntu.edu.tw/~cjlin/libsvmtools/datasets/}}, including \texttt{australian} ($W \in \mathbb{R}^{690 \times 15}$), \texttt{mushrooms} ($W \in \mathbb{R}^{8124 \times 113}$), \texttt{phishing} ($W \in \mathbb{R}^{11055 \times 69}$), and \texttt{ijcnn} ($W \in \mathbb{R}^{49990 \times 23}$). The columns of $W$ are rescaled to $[-1,1]$ prior to the experiment. We set the tuning parameters to $\lambda_1 = \lambda_2 = \tfrac{1}{10}$, $\beta = \tfrac{1}{2}$, and $\lambda_3 = 10^3$.

We set the step sizes, extrapolation parameter, and momentum parameter in ACV to the values suggested in Algorithm \ref{alg:acv4}. For the test of the traditional Condat--V\~u algorithm, we use parameters suggested by theory \cite[Eq.\! 48]{CPrates}, labelled ``book'' parameters, and tuned parameters. To tune the parameters, we search the set $i \times 10^{-j} \taustep_{\textnormal{acc}}$, $i \in \{1, 2, \cdots, 10\}$ and $j \in \mathbb{N}$ to find the largest primal step size that ensures convergence, where $\taustep_{\textnormal{acc}}$ is the primal step sizes suggested by our theory for ACV. We find a wide range of dual step sizes and extrapolation parameters can be used without affecting performance, but the algorithm is sensitive to the primal step size value. Figure \ref{fig:cv-elastic-net} shows the results of these experiments along with the convergence rates predicted by theory.

From Figure \ref{fig:cv-elastic-net}, we see that Accelerated Condat--V\~u significantly outperforms its non-accelerated variant. This large discrepancy exists because the ratio $L / \|A\|_{\textnormal{op}}$ is large---for example, in the \texttt{mushrooms} dataset, $L = 6.1 \times 10^4$---and this is the regime where Accelerated Condat--V\~u offers the most improvement over the traditional Condat--V\~u algorithm.

\subsection{Fused elastic net without smoothing}

In this section, we consider problem \eqref{eq:ggfen} with $J(Fx)$ replaced by $\|F x\|_1$ (equivalently, we set $\lambda_3 \to \infty$ in the definition of $J$ \eqref{eq:J}). With this objective, $f^*$ is no longer strongly convex, so we must apply Algorithm \ref{alg:acv2}. We keep all the tuning parameters set to the values in the previous section. For ACV, we use the recommended parameters suggested by Theorem \ref{thm:cvacc}. We find that ACV converges at a fast rate using the constant parameter settings of the warm-up phase, so we set $T_0 = \infty$ in each of these experiments. For the traditional Condat--V\~u algorithm, we use the step sizes suggested in \cite[Eq. 40]{CPrates} as the ``book'' parameter settings. For the ``tuned'' parameter settings, we use the book settings, but scale the initial dual step size $\gamma_0$ by $10^j$ for $j \in \{-5,-4,\cdots,4,5\}$ and record the instance that provides the fastest convergence. Figure \ref{fig:cv-elastic-net-nonsmooth} shows the results of these experiments. worse than $(1 + \sqrt{\kappa_P})^{-T}$. In every case, ACV significantly outperforms traditional Condat--V\~u.

\begin{figure}[!t]
\centering
 \includegraphics[width=0.99\linewidth]{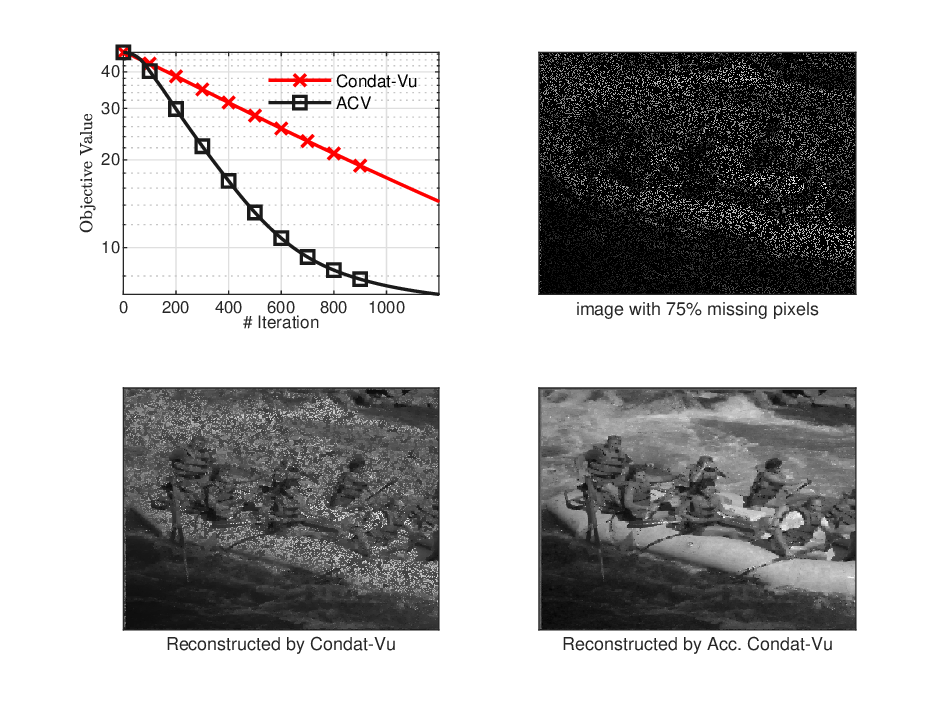} 
\caption{
Results for the inpainting experiment (example 1). Here we compare the performance of CV with fine-tuned step-sizes and ACV with the step-size rule \eqref{step_size_default}. In the second row we show the images reconstructed by the algorithms at the 1200th iteration.}
\label{fig:ip1}
\end{figure}

\begin{figure}[!t]
\centering
 \includegraphics[width=0.99\linewidth]{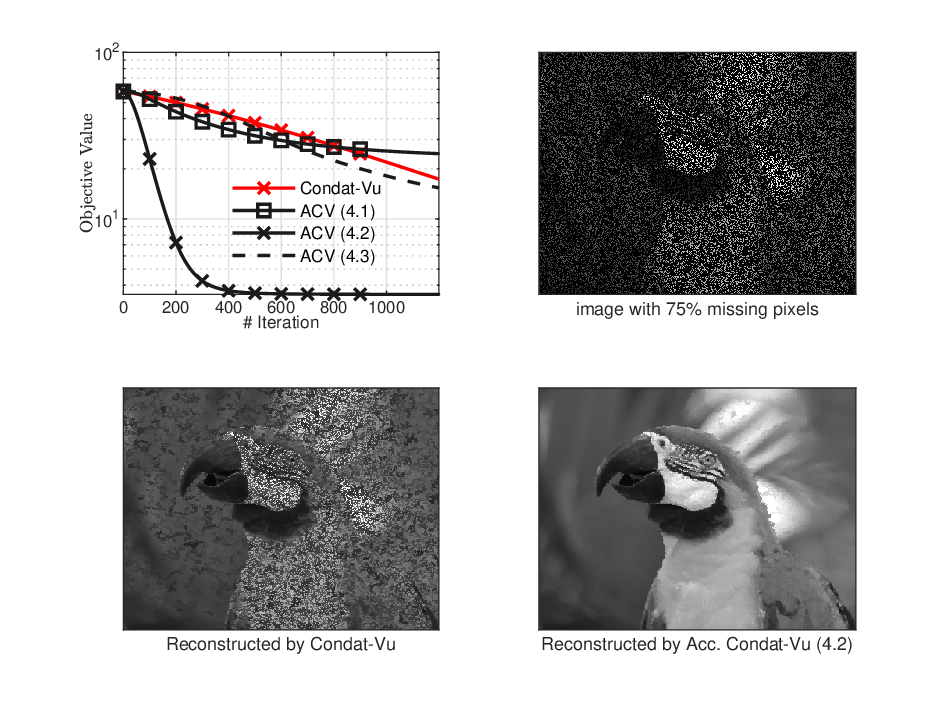} 
\caption{
 Results for the inpainting experiment (example 2) with strong-convexity on $g$. Here we compare the performance of CV with fine-tuned step-sizes and ACV with all three step-size rules. We can observe that the ACV presented in Alg. 4.2 provide significant acceleration. When there is significant strong-convexity in the objective, it is crucial for the ACV to select the step-size rule which exploits the strong-convexity. In the second row we show the images reconstructed by the algorithms at the 1200th iteration.}
\label{fig:ip2}
\end{figure}

\begin{figure}[!t]
\centering
 \includegraphics[width=0.99\linewidth]{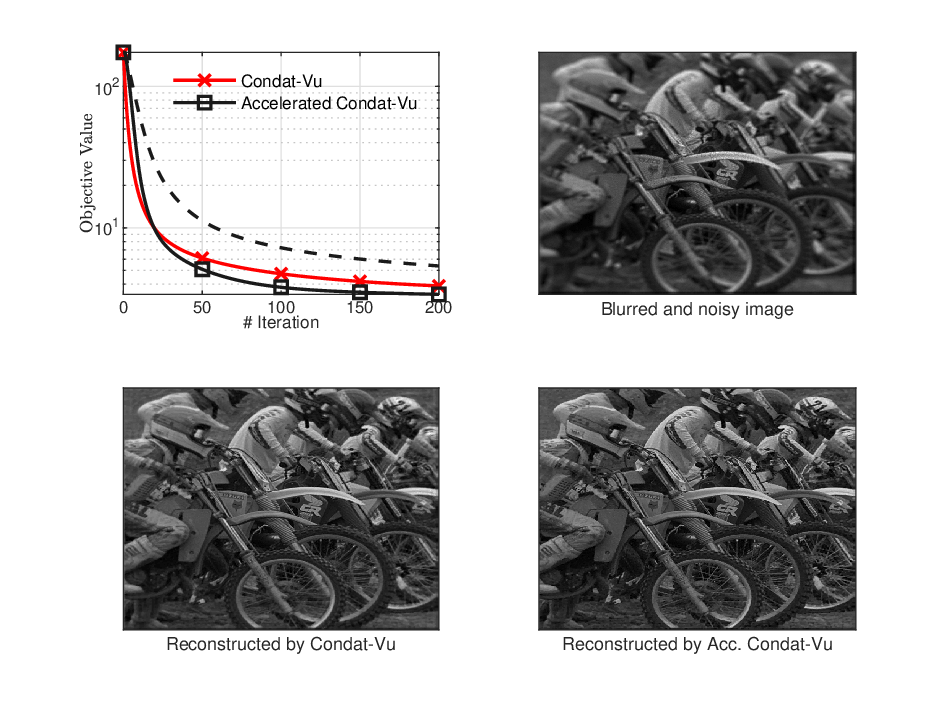} 
\caption{
Results for the deblurring experiment (example 1). Here we compare the performance of CV with fine-tuned step-sizes and ACV with the step-size rule \eqref{step_size_default}. Here the solid black curve records the performance of ACV with the rescaling trick \eqref{res}, while the dash curve is for ACV without using the rescaling trick (same for the following figures). In the second row we show the images reconstructed by the algorithms at the 200th iteration.}
\label{fig:deb}
\end{figure}

\begin{figure}[!t]
\centering
 \includegraphics[width=0.99\linewidth]{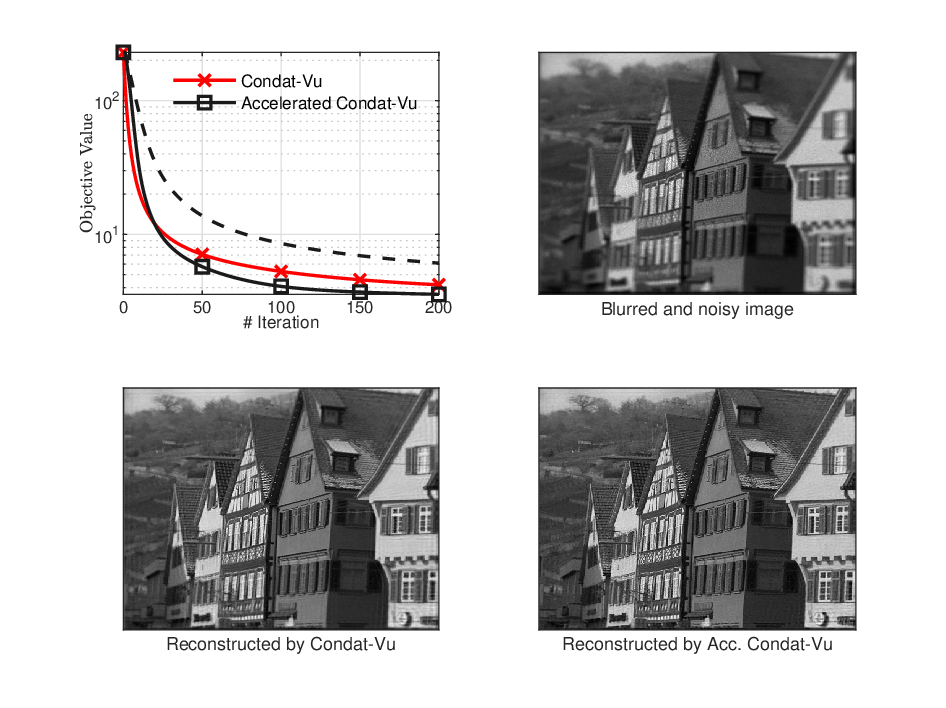} 
\caption{
Results for the deblurring experiment (example 2). Here we compare the performance of CV with fine-tuned step-sizes and ACV with the step-size rule \eqref{step_size_default}. Here the solid black curve records the performance of ACV with the rescaling trick \eqref{res}, while the dash curve is for ACV without using the rescaling trick. In the second row we show the images reconstructed by the algorithms at the 200th iteration.}
\label{fig:deb2}
\end{figure}

\begin{figure}[!t]
\centering
 \includegraphics[width=0.99\linewidth]{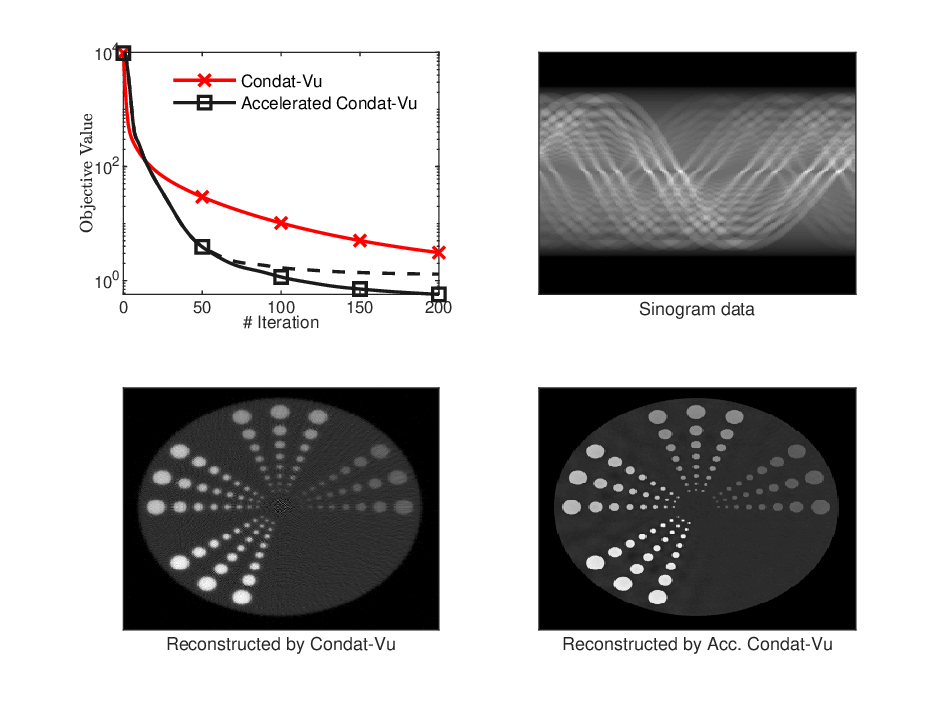} 
\caption{
Results for X-ray CT experiment (example 1). Here we compare the performance of CV with fine-tuned step-sizes and ACV with the step-size rule \eqref{step_size_default}. Here the solid black curve records the performance of ACV with the rescaling trick \eqref{res}, while the dash curve is for ACV without using the rescaling trick. In the second row we show the images reconstructed by the algorithms at the 200th iteration. }
\label{fig:CT}
\end{figure}

\begin{figure}[!t]
\centering
 \includegraphics[width=0.99\linewidth]{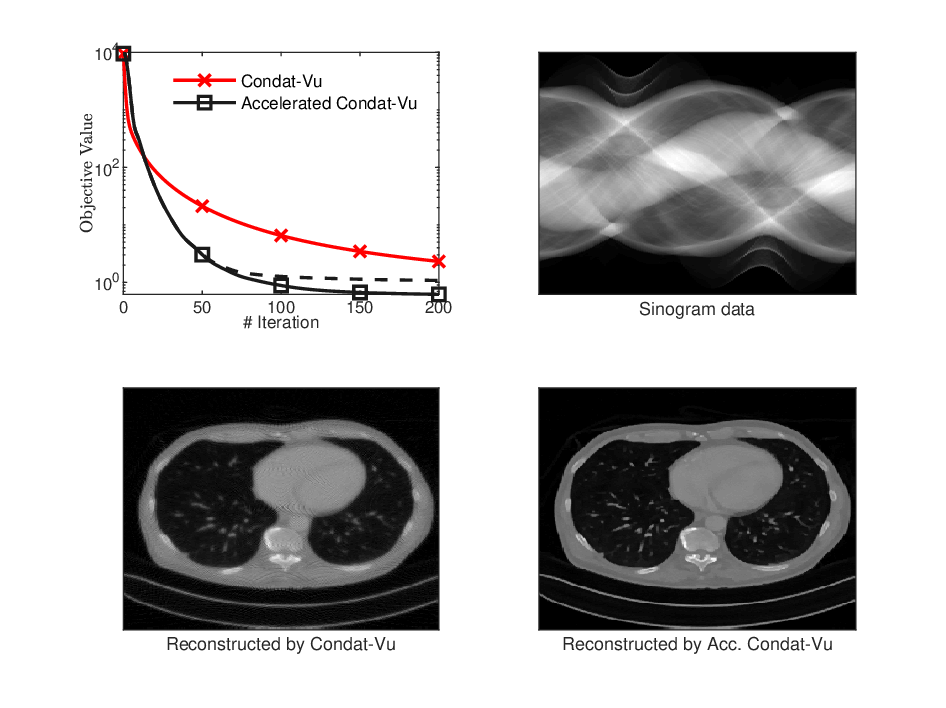} 
\caption{
Results for X-ray CT experiment (example 2). Here we compare the performance of CV with fine-tuned step-sizes and ACV with the step-size rule \eqref{step_size_default}. Here the solid black curve records the performance of ACV with the rescaling trick \eqref{res}, while the dash curve is for ACV without using the rescaling trick. In the second row we show the images reconstructed by the algorithms at the 200th iteration.}
\label{fig:CT2}
\end{figure}

\subsection{Image Deblurring}
We next consider deblurring task where we seek to minimize:
\begin{equation}
\label{eq:tv_db}
    \min_x \quad \frac{1}{2} \|Mx - b\|^2 + \lambda_1  \|D x\|_1 + \iota_{\mathbb{R}_+}(x),
\end{equation}
where $b$ is the blurred noisy image, with a non-uniform out-of-focus blur across the image, while the blurring effect is increasingly more severe from the center to the edge of the image. Here, the forward operator $M$ encodes the blurring process, $\lambda > 0$ is the regularization parameter, $D$ is the forward difference matrix, $\iota_C$ is the indicator function of set $C$, and $\mathbb{R}_+$ is the non-negative cone. We set $h(x) = \frac{1}{2} \|Mx - b\|^2$, $g = \iota_{\mathbb{R}_+}(x)$, and $f$ to be the $\ell_1$-norm with conjugate $f^*(y) = \iota_{\|\cdot\|_{\infty} \le \lambda_1}$. We then reformulate the primal problem \eqref{eq:tv_db} to the following saddle point problem:
\begin{equation}
    [x^\star, y^\star] = \min_x \max_y \frac{1}{2}\|Mx - b\|_2^2 + \iota_{\mathbb{R}_+}(x) + y^TDx - \iota_{\|\cdot\|_{\infty} \le \lambda_1}(y),
\end{equation}
and solve it using Algorithm \ref{alg:acv} with the recommended step-sizes, with comparison to the fine-tuned {Condat-V\~u} scheme. However, in this example we also have the issue of $L \approx \|D\|$ hence we do not immediately observe acceleration over {Condat-V\~u} with fine-tuned step-sizes. Inspired by the theoretical result we introduce a tuning parameter $\rho$ to scale down the linear operator $D$ and meanwhile scaling-up the indicator constraint accordingly while keeping the optimization problem equivalent to the original one:
\begin{equation}\label{res}
    [x^\star, y^\star] = \min_x \max_y \frac{1}{2}\|Mx - b\|_2^2 + \iota_{\mathbb{R}_+}(x) + \frac{1}{\rho_1}y^TDx - \iota_{\|\cdot\|_{\infty} \le \lambda_1\rho_1}(y),
\end{equation}
Here we choose $\rho=100$ for convenience. This scaling makes a convergence trade-off between the dependence of $\|D\|$ (which is believed to be tight) and the dependence of $\max_{y \in \mathcal{B}_2} \|y - y_0\|_2$ (which is usually believed to be loose). This transformation can be applied to improve the numerical performance for our acceleration scheme, but we do not observe such improvement for the original {Condat-V\~u} algorithm. We report the results of this non-uniform deblurring experiment in Figure \ref{fig:deb} and Figure \ref{fig:deb2}, where we can observe that our scheme can achieve faster convergence rates compared to original {Condat-V\~u} if we perform the rescaling trick to suppress the negative impact of $L \approx \|D\|$ on acceleration.


\subsection{Image Inpainting}
{  We then consider image inpainting task where we seek to inpaint images with significant amounts of missing pixels ($75\%$ of pixels are missing). we consider two scenarios of objective functions. The first scenario would be of the same form as our image deblurring experiments presented in the previous section, with the forward operator replaced by the binary mask. We present the results of this scenario in Figure \ref{fig:ip1}, where we observe again the improvement in terms of convergence speed over classical Condat-Vu scheme.

The second scenario we consider in an additional $\ell_2$ regularization, and test all three sets of parameter choices for our ACV. The objective now reads:
\begin{equation}
\label{eq:tv_ip_sc}
    \min_x \quad \frac{1}{2} \|Mx - b\|^2 + \lambda_1  \|D x\|_1 + \iota_{\mathbb{R}_+}(x) + \frac{\mu_g}{2}\|x\|_2^2,
\end{equation}
where we set $g(x) = \iota_{\mathbb{R}_+}(x) + \frac{\mu_g}{2}\|x\|_2^2$ with $\mu_g = 0.05$. We present our results in Figure \ref{fig:ip2}. The step-size rule of ACV presented in Algorithm 4.2 is tailored for exploting the strong-convexity in $g$ and we can observe that in this example this rule for ACV indeed provides significant acceleration over standard Condat-Vu. On the other hand, the step-size rule in Algorithm 4.2 achieve acceleration initially but plateau later due to ignoring strong-convexity, while rule 4.3 assumes $f^*$ is also strongly-convex, which is not the case in this example, although it is still providing some minor acceleration if we heuristically input a tuned value for $\mu_{f^*}$ in \eqref{scx_recom_step_sizes}. This example demonstrates that if the objective is explicitly strongly-convex, it is recommended to adapt the step-size rule of ACV to exploit such structure to ensure fastest convergence.

}

\subsection{X-ray CT reconstruction}

 In this section we demonstrate our numerical experiments on model-based X-ray CT reconstruction. The underlying optimization problem takes the same form as in the deblurring task \eqref{eq:tv_db}, where in this case the forward operator is the radon transform for a fan beam CT geometry. In this experiment we simulate equally spaced 180 views with 256 rays per view, while the measurement data is corrupted with a small amount of Poisson noise. Here we try the default $\rho_1 = 1$ (denoted by the dash line) and the adjustment $\rho_1 = 0.01$ for our accelerated {Condat-V\~u} scheme. Note that in this example we have $L$ significantly larger than the operator norm of $D$, hence we should expect a clear acceleration of our scheme over the original {Condat-V\~u} according to the theory. Indeed, our results in Figure \ref{fig:CT} demonstrate such acceleration effect, where we can clearly observe that in 200 iterations the our scheme has reached a much higher quality reconstruction comparing to the original {Condat-V\~u}'s reconstruction.

\section{Conclusion}

The fact that the Condat--V\~u algorithm reduces to PGD rather than the accelerated variant makes its suboptimality obvious. Condat--V\~u is restricted to interpolate between the convergence rates of PGD and PDHG depending on relative importance of each function in the composite objective. In regimes where the Condat--V\~u algorithm approximates PGD, its performance is poor. The proposed Accelerated Condat--V\~u algorithm solves this problem. Its performance interpolates between the performance of APGD and PDHG, exhibiting superior performance over this entire spectrum of problems, but especially when the rate of APGD dominates. In addition to the empirical success of ACV in experiments, we also show that ACV achieves optimal convergence rates no matter the strong convexity or smoothness of the composite functions.

\appendix

\section{Visualising the Relationships Among Primal-Dual Algorithms}
\label{app:fig}

In Figure \ref{fig:alg} we present a graph to help the readers in understanding the relationships among the primal-dual algorithms in the literature.

\begin{figure}[h]
\centering
\begin{tikzpicture}[
    sibling distance        = 5.5em,
    level distance          = 10em,
    edge from parent/.style = {draw, -latex},
    every node/.style       = {font=\footnotesize},
    sloped,
    env/.style = {
    shape=rectangle
    ,  rounded corners
    , draw=black!90
    ,  top color=blue!0
    ,  bottom color=blue!20
    , inner sep=1.5mm
    , outer sep=0mm
    }
  ]
\node [env] (topnode) {Davis--Yin Splitting}
    child { node [env] {PDDY/AFBA} }
    child { node [env] {PDFP}
        child { node [env] {PAPA} 
            edge from parent node [above]
                  {$g \equiv \iota_C$}}
          edge from parent [draw=none]
                  }
    child { node [env] {PDDY/AFBA}
        edge from parent [draw=none] }
    child { node [env] {PD3O} 
        child { node [env] (PAPC) {PAPC}
            edge from parent node [above]
              {$g \equiv 0$} }
        child { node [env] (CP) {PDHG}
        child { node [env] {A-H}
        child { node [env] (PGD) {PGD} 
            edge from parent node [above]
            {$f \circ A$ smooth} } 
          edge from parent node [above]
                  {$\theta = 0$} }
        child { node [env] (APGD) {APGD}
          edge from parent[dotted] node [above]
                  {$f \circ A$ smooth} }
        edge from parent node [above] {$h \equiv 0$}} }
    child { node [env] {Condat--V\~u} }
    child { node [env] [right=0mm,top color=red!0,bottom color=red!30] (ACV) {\textbf{ACV (This Work)}}
        child { node [env] {APD} 
            edge from parent node [above]
                  {$g \equiv 0$}}
        edge from parent [draw=none] };
\draw [-latex] (ACV) to[out=100,in=80] node[above]{$\alpha_k = 0$}(topnode-5);
\draw [-latex] (topnode-3) to[out=0,in=180] node[above]{ }(topnode-4);
\draw [-latex] (ACV) -- node [left=10mm,above] {$h \equiv 0$}(CP);
\draw [-latex] (ACV) -- node [left=10mm,above] {$f \circ A \equiv 0$} (APGD);
\draw [-latex] (topnode-5) -- (CP);
\draw [-latex] (topnode-2) -- (PAPC);
\draw [-latex] (PAPC) -- node [left=10mm,below] {$A \equiv I_d$}(PGD);
\draw [-latex] (topnode-3) -- (PAPC);
\draw [-latex] (topnode-3) to[out=80,in=100] (topnode-5);
\draw [-latex] (topnode-4) to[out=-10,in=0] node[above]{$f \circ A \equiv 0$} (PGD);
\draw [-latex] (topnode-5) to[out=-30,in=0] node[above]{$f \circ A \equiv 0$} (PGD);
\draw [-latex] (topnode-1) to[out=-100,in=180] node[above]{$f \circ A \equiv 0$}(PGD);
\end{tikzpicture}
\caption{A summary of primal-dual algorithms and their relationships. }
\label{fig:alg}
\end{figure}
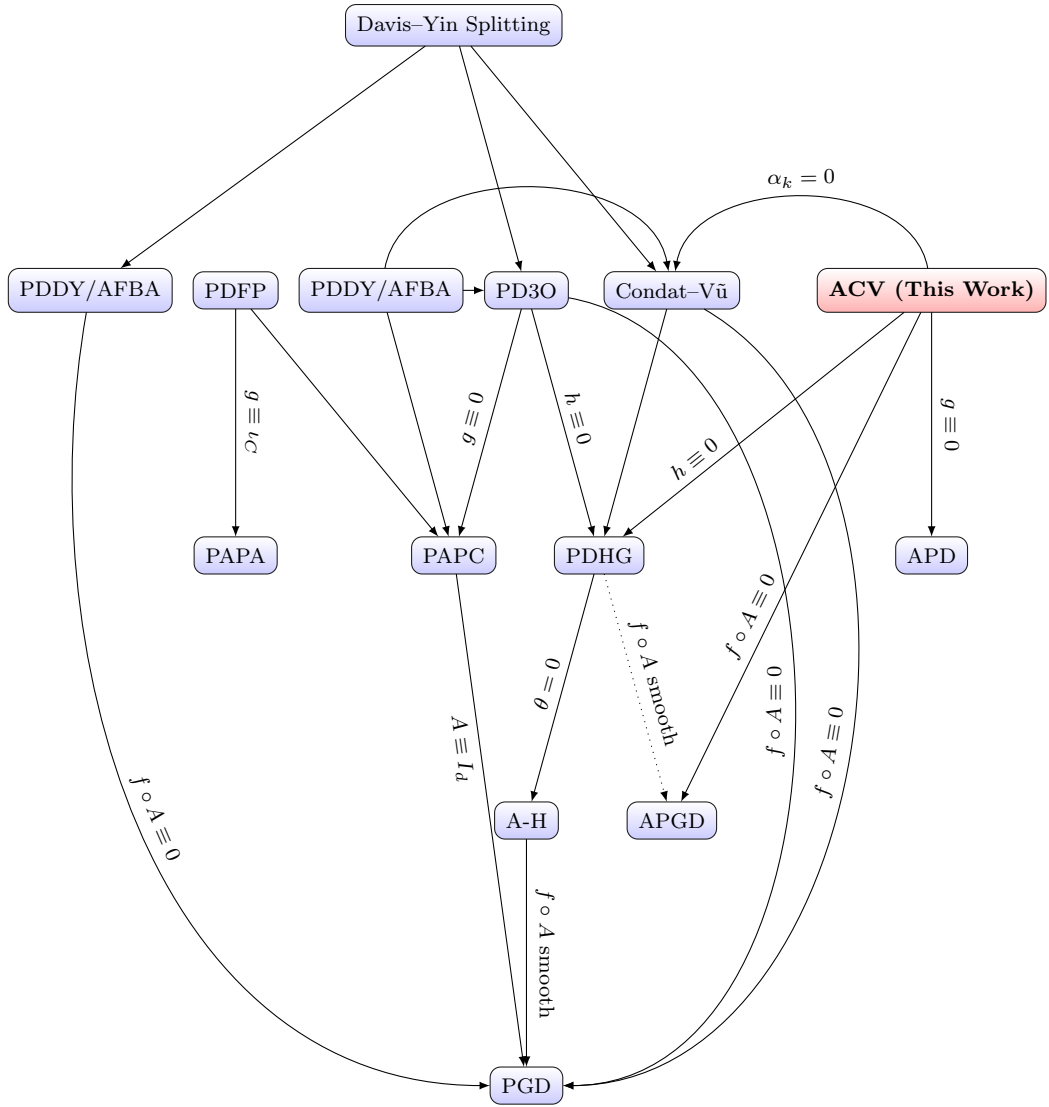

\bibliographystyle{plain}
\bibliography{main.bib}

\end{document}